\documentclass[12pt]{amsart}

\usepackage{amssymb, amsmath, amsthm}
\usepackage[margin=1in]{geometry}
\usepackage{tikz}

\allowdisplaybreaks

\numberwithin{equation}{section}

\theoremstyle{plain}
\newtheorem{thm}{Theorem}[section]
\newtheorem{prop}[thm]{Proposition}
\newtheorem{lem}[thm]{Lemma}
\newtheorem{cor}[thm]{Corollary}

\theoremstyle{definition}
\newtheorem{defn}[thm]{Definition}
\newtheorem{rem}[thm]{Remark}

\newcommand{\ichi}{\mathbf{1}}
\newcommand{\C}{\mathbb{C}}
\newcommand{\N}{\mathbb{N}}
\newcommand{\R}{\mathbb{R}}

\newcommand{\Z}{\mathbb{Z}}
\newcommand{\calA}{\mathcal{A}}
\newcommand{\calB}{\mathcal{B}}
\newcommand{\calBform}{\mathcal{B}^{\, \mathrm{form}\, }}

\newcommand{\calF}{\mathcal{F}}

\newcommand{\calM}{\mathcal{M}}
\newcommand{\calS}{\mathcal{S}}
\newcommand{\supp}{\mathrm{supp}\, }

\newcommand{\card}{\mathrm{card}\, }

\newcommand{\K}{\mathbb{K}}

\newcommand{\op}{\mathrm{Op}}

\newcommand{\vs}{\vspace{12pt}}

\begin{document}
\title[Bilinear pseudo-differential operators
of $S_{0,0}$-type]
{Boundedness  
of bilinear pseudo-differential operators 
of $S_{0,0}$-type in 
Wiener amalgam spaces and in Lebesgue spaces} 

\author[T. Kato]{Tomoya Kato}
\author[A. Miyachi]{Akihiko Miyachi}
\author[N. Tomita]{Naohito Tomita}

\address[T. Kato]
{Division of Pure and Applied Science, 
Faculty of Science and Technology, Gunma University, 
Kiryu, Gunma 376-8515, Japan}

\address[A. Miyachi]
{Department of Mathematics, 
Tokyo Woman's Christian University, 
Zempukuji, Suginami-ku, Tokyo 167-8585, Japan}

\address[N. Tomita]
{Department of Mathematics, 
Graduate School of Science, Osaka University, 
Toyonaka, Osaka 560-0043, Japan}

\email[T. Kato]{t.katou@gunma-u.ac.jp}
\email[A. Miyachi]{miyachi@lab.twcu.ac.jp}
\email[N. Tomita]{tomita@math.sci.osaka-u.ac.jp}

\date{\today}

\thanks{This work was supported by JSPS KAKENHI, 
Grant Numbers 
20K14339 (Kato), 20H01815 (Miyachi), and 20K03700 (Tomita).}

\begin{abstract}
We extend and improve 
the known results 
about the boundedness 
of the bilinear pseudo-differential 
operators with symbols in the
bilinear H\"ormander class $BS^{m}_{0,0}(\R^n)$.  
We consider wider classes of symbols 
and improve estimates 
for the corresponding operators. 
A key idea is to consider the 
operators in 
Wiener amalgam spaces. 
\end{abstract}

\maketitle

\section{Introduction}\label{Introduction}
\subsection{Background}

For a bounded measurable 
function $\sigma = \sigma (x, \xi_1, \xi_2)$ on $\R^n \times \R^n \times \R^n$,
the bilinear pseudo-differential operator
$T_{\sigma}$ is defined by
\[
T_{\sigma}(f_1,f_2)(x)
=\frac{1}{(2\pi)^{2n}}
\int_{\R^n \times \R^n}e^{i x \cdot(\xi_1+\xi_2)}
\sigma(x,\xi_1,\xi_2)\widehat{f_1}(\xi_1)
\widehat{f_2}(\xi_2)\, d\xi_1 d\xi_2, 
\quad x \in \R^n, 
\]
for $f_1,f_2 \in \calS(\R^n)$. 
The function $\sigma$ is called the symbol of the operator 
$T_{\sigma}$. 
In particular, 
if the symbol $\sigma$ is independent of $x$, that is 
 $\sigma = \sigma (\xi_1, \xi_2)$, 
then $\sigma$ is called the Fourier multiplier and 
$T_{\sigma}$ is called the bilinear Fourier multiplier operator.
The study of bilinear pseudo-differential operators
started from a problem of the Calder\'on commutators
in the early works of Coifman and Meyer \cite{CM-Ast, CM-AIF}
and is continued by a lot of researchers.

In this paper, we consider the boundedness of 
the bilinear pseudo-differential operators. 
We shall 
use the following terminology. 
Let $X_1,X_2$, and $Y$ be function spaces on $\R^n$ 
equipped with quasi-norms 
$\|\cdot \|_{X_1}$, $\|\cdot \|_{X_2}$, 
and $\|\cdot \|_{Y}$, 
respectively.  
If there exists a constant $C$ such that 
\begin{equation}\label{eqboundedness-X_1X_2Y}
\|T_{\sigma}(f_1,f_2)\|_{Y}
\le C \|f_1\|_{X_1} \|f_2\|_{X_2} 
\;\;
\text{for all}
\;\;
f_1\in \calS \cap X_1  
\;\;
\text{and}
\;\;
f_2\in \calS \cap X_2,  
\end{equation}
then, 
with a slight abuse of terminology, 
we say that 
$T_{\sigma}$ is bounded from 
$X_1 \times X_2$ to $Y$ 
and write 
$T_{\sigma}: X_1 \times X_2 \to Y$.  
The smallest constant $C$ of 
\eqref{eqboundedness-X_1X_2Y} 
is denoted by 
$\|T_{\sigma}\|_{X_1 \times X_2 \to Y}$. 
If $\calA$ is a class of symbols,  
we denote by $\mathrm{Op}(\calA)$
the class of all bilinear 
operators $T_{\sigma}$ 
corresponding to $\sigma \in \calA$. 
If $T_{\sigma}: X_1 \times X_2 \to Y$ 
for all $\sigma \in \calA$, 
then we write 
$\mathrm{Op}(\calA) 
\subset B (X_1 \times X_2 \to Y)$.

The bilinear H\"ormander symbol class
$BS^m_{\rho,\delta}=
BS^m_{\rho,\delta}(\R^n)$,
$m \in \R$, $0 \le \rho, \delta \le 1$,
consists of all
$\sigma(x,\xi_1,\xi_2) \in C^{\infty}(\R^n \times \R^n \times \R^n)$
such that
\[
\left|
\partial^{\gamma }_x 
\partial^{\beta_1}_{\xi_1} 
\partial^{\beta_2}_{\xi_2} 
\sigma(x,\xi_1,\xi_2)
\right|
\le C_{\gamma,\beta_1,\beta_2}
(1+|\xi_1|+|\xi_2|)^{m+\delta|\gamma|-\rho(|\beta_1|+|\beta_2|)}
\]
for all multi-indices
$\gamma,\beta_1,\beta_2 \in (\N_0)^n
=\{0, 1, 2, \dots \}^n$. 
The 
class $BS^m_{\rho,\delta}$ was introduced by 
B\'enyi--Maldonado--Naibo--Torres 
\cite{BMNT} 
and investigated by 
B\'enyi--Bernicot--Maldonado--Naibo--Torres 
\cite{BBMNT}. 
See these papers for the 
basic properties of the class 
$BS^m_{\rho,\delta}$  
including 
symbolic calculus, duality, and interpolation.

We shall recall boundedness properties of
bilinear pseudo-differential operators in
the class $\op(BS^m_{\rho,\delta})$.
If $\rho=1$ and $\delta < 1$, 
the situation is similar to the linear case.
In this case, 
the bilinear Calder\'on--Zygmund theory 
developed by Grafakos--Torres \cite{GT}
implies that the bilinear operator $T_{\sigma}$ with 
$\sigma$ in $BS^{0}_{1,\delta}$
is bounded from $L^{p_1} \times L^{p_2}$ to $L^p$
for $1 < p_1, p_2 \le \infty$ and $1/p = 1/p_1 + 1/p_2>0$; 
see also Coifman--Meyer \cite{CM-Ast, CM-AIF}, 
B\'enyi--Torres \cite{BT-2003}, and 
B\'enyi--Maldonado--Naibo--Torres \cite{BMNT}. 
Here, it should be remarked that 
the condition $1/p = 1/p_1 + 1/p_2$ is necessary,
since constant functions belong to 
$BS^{0}_{1,\delta}$, $\delta < 1$, 
and the operator $T_{\sigma}$ for $\sigma \equiv 1$ 
gives the pointwise product of two functions, 
which admits the estimate 
$L^{p_1} \times L^{p_2} \to L^p$ 
only for $1/p = 1/p_1 + 1/p_2$. 
However, in the case $\rho < 1$,
the situation of the bilinear operators 
varies considerably
from the linear ones.
To see this, let us consider the special case $\rho=\delta=0$. 
B\'enyi--Torres \cite{BT-2004} proved that 
for any $1 \le p_1, p_2, p < \infty$ satisfying 
$1/p = 1/p_1 + 1/p_2$
there exists a multiplier in $BS^0_{0,0}$ 
for which the corresponding bilinear 
Fourier multiplier operator is not bounded 
from $L^{p_1} \times L^{p_2}$ to $L^p$.
In particular 
the operators in $\op(BS^{0}_{0,0})$ are not always
bounded from $L^{2} \times L^{2}$ to $L^{1}$, 
which presents a remarkable contrast to the well-known 
Calder\'on--Vaillancourt theorem \cite{CV} 
for linear pseudo-differential operators. 
In what follows we shall be interested in 
the case $\rho=\delta=0$.
(For the general case
$0 \le \delta \le \rho < 1$, 
see, for instance, 
\cite{BBMNT, KS, 
MRS-2014, MT-2018 AIFG, MT-2019 JPDOA, Nai-2015}
and the references therein.)

From the result of 
B\'enyi--Torres \cite{BT-2004} mentioned above, 
in order to have 
$\op(BS^{m}_{0,0}) \subset B( L^{p_1} \times L^{p_2} \to L^{p})$,
the order $m$ must be negative. 
So, to find the condition on $m$ is a problem. 
After the works of 
Michalowski--Rule--Staubach \cite{MRS-2014}
and 
B\'enyi--Bernicot--Maldonado--Naibo--Torres \cite{BBMNT}, 
in the paper \cite{MT-2013}, 
the second and the third named authors 
of the present paper 
proved that, for 
$0 < p_1, p_2, p \le \infty$ satisfying 
$1/p = 1/p_1 + 1/p_2$,  
the relation 
\begin{equation} \label{eqbddp1p2p}
\op(BS^{m}_{0,0}) \subset
B (h^{p_1} \times h^{p_2} \to h^{p}), 
\end{equation}
where $h^{p_1}$ (resp.\  $h^{p_2}$, $h^{p}$) 
should be replaced by $bmo$ 
if $p_1=\infty$ (resp.\ $p_2=\infty$, $p=\infty$), 
holds 
if and only if $m \le m(p_1,p_2,p)$ 
with 
\begin{equation}\label{eqcriticalorder}
m(p_1,p_2,p)
= 
\min (n/2, n/p) 
- 
\max (n/2, n/p_1) 
-\max (n/2, n/p_2).
%
\end{equation} 
(The definitions of the spaces 
$h^r$ and $bmo$ will be given in the next section.)

Recently, there appeared several works 
that consider symbol classes of $S_{0,0}$-type 
different from $BS^{m}_{0,0}(\R^n)$. 
Notice that, by the closed graph theorem, the assertion
\eqref{eqbddp1p2p} is equivalent to the estimate 
\begin{equation*}
\left\| 
T_{\sigma} \right\|_{h^{p_1}\times h^{p_2} \to 
h^{p}} 
\lesssim 
\sup_{|\alpha|\le M} 
\left\| 
(1+|\xi_1|+|\xi_2|)^{-m}\, 
\partial^{\alpha}_{x,\xi_1, \xi_2}
\sigma (x, \xi_1, \xi_2)
\right\|_{L^{\infty}(\R^n \times \R^n\times \R^n)}
\end{equation*}
with some $M \in \N$. 
Some recent works consider estimates of 
the operator norm of $T_{\sigma}$ 
in terms of different kinds of 
norms of $\sigma$. 
The first one of such researches is 
Grafakos--He--Honz\'ik \cite{GHH}, 
in which the authors 
considered the $L^2 \times L^2 \to L^1$ 
estimate for bilinear Fourier multiplier operators 
and proved the estimate 
\begin{equation*}
\left\| 
T_{\sigma} \right\|_{L^{2}\times L^{2} \to 
L^{1}} 
\lesssim 
\bigg(\sup_{|\alpha|\le M} 
\left\| 
\partial^{\alpha}_{\xi_1, \xi_2}
\sigma (\xi_1, \xi_2)
\right\|_{L^{\infty}(\R^n \times \R^n)} \bigg)^{1/5}
\bigg( \left\|
\sigma (\xi_1, \xi_2)
\right\|_{L^{2} (\R^n \times \R^n)}
\bigg)^{4/5}; 
\end{equation*}
see \cite[Corollary 8]{GHH}. 
Grafakos--He--Slav\'ikov\'a \cite[Theorem 1.3]{GHS} 
generalized the above result by showing  
the estimate 
\begin{equation}\label{eq221-4-epsilon}
\left\| 
T_{\sigma} \right\|_{L^{2}\times L^{2} \to 
L^{1}} 
\lesssim 
\bigg(\sup_{|\alpha|\le M} 
\left\| 
\partial^{\alpha}_{\xi_1, \xi_2}
\sigma (\xi_1, \xi_2)
\right\|_{L^{\infty}(\R^n \times \R^n)} \bigg)^{1- q /4}
\bigg( \left\|
\sigma (\xi_1, \xi_2)
\right\|_{L^{q} (\R^n \times \R^n)}
\bigg)^{q/4}  
\end{equation}
for any $1\le q <4$. 
The authors of 
\cite{GHH} and \cite{GHS} 
applied the above estimates to prove 
boundedness of 
bilinear singular integral operators with rough kernels.  
In \cite{KMT-arxiv}, 
the present authors considered symbols that satisfy 
\begin{equation}\label{eqcondition-for-BSW00}
\left|
\partial^{\alpha}_{x, \xi_1, \xi_2} 
\sigma(x,\xi_1,\xi_2)
\right|
\le C_{\alpha} 
W (\xi_1, \xi_2)
\end{equation}
for all multi-indices $\alpha \in (\N_0)^{3n}$ 
and, 
under certain mild condition on $W$, 
proved the estimate   
\begin{equation}\label{eq221-calB}
\left\| 
T_{\sigma} \right\|_{L^{2}\times L^{2} \to 
L^{1}} 
\lesssim 
\|W\|_{\calB}, 
\end{equation}
where $\|W\|_{\calB}$ denotes the smallest 
constant $c$ such that the inequality 
\begin{equation*}
\sum_{\nu_1, \nu_2 \in \Z^n} 
W(\nu_1, \nu_2) 
A(\nu_1) 
B(\nu_2)
C(\nu_1+\nu_{2})
\le c
\|A\|_{\ell^{2} (\Z^n) } 
\|B\|_{\ell^{2} (\Z^n) } 
\|C\|_{\ell^{2} (\Z^n) } 
\end{equation*}
holds for all nonnegative 
functions $A,B,C$ on $\Z^n$. 
From the estimate 
\eqref{eqcondition-for-BSW00}-\eqref{eq221-calB}, 
it is possible to derive the estimate 
\begin{equation}\label{eq221-4infty}
\|T_{\sigma}\|_{L^{2} \times L^{2} \to L^{1} } 
\le 
c 
\sum_{|\alpha|\le K} 
\left\|
\sup_{x\in \R^n} \left| 
\partial_{x,\xi_1, \xi_2}^{\alpha} 
\sigma (x,\xi_1, \xi_2)
\right| 
\right\|_{L^{4,\infty}_{\xi_1, \xi_2} (\R^n \times \R^n)}  
\end{equation}
with $K$ a sufficiently large positive integer 
(see the argument in Subsection 
\ref{proofofmain-cor_1} of the present paper). 
Slav\'ikov\'a \cite{Slavikova}  
proved the estimate 
\eqref{eq221-4infty} 
for the Fourier multiplier case  
in a refined form. 
In \cite{KMT-arxiv-2}, 
the present authors 
pushed forward the method of 
\cite{KMT-arxiv} 
and, utilizing 
the $L^2$-based amalgam spaces,  
generalized the estimates 
\eqref{eq221-calB} and $\eqref{eq221-4infty}$ 
to the case 
$L^{p_1}\times L^{p_2} \to L^p$ 
with 
$2 \le p_1, p_2 \le \infty$ and  
$1/2 \le 1/p\le 1/p_1 + 1/p_2$;   
see \cite[Theorem 6.1]{KMT-arxiv-2}.  
An attempt to generalize 
\eqref{eq221-4-epsilon} and 
\eqref{eq221-4infty} to the case 
$L^{p_1}\times L^{p_2} \to L^p$ 
for all $p_1, p_2 \in [1, \infty]$ 
was made by 
Buri\'ankov\'a--Grafakos--He--Honz\'ik 
\cite{BGHH}; 
they considered 
the special multiplier defined by 
\[
m_{E, \Phi}(\xi_1, \xi_2)
=
\sum_{(\nu_1, \nu_2)\in E}
\Phi (\xi_1 - \nu_1, \xi_2 - \nu_2), 
\]
where 
$\Phi$ is a $C^{\infty}$ function 
on $\R^{2n}$ supported in the ball 
$\{ |\zeta| \le 1/20\}$ 
and $E\subset \Z^n \times \Z^n$, 
and proved estimates of the form 
\begin{equation}\label{eqlattice-bump-alpha}
\left\| 
T_{m_{E, \Phi}} 
\right\|_{L^{p_1}\times L^{p_2} \to 
L^{p}} 
\lesssim 
(\card E)^{\alpha (p_1, p_2, p)}  
\end{equation}
for $1 \le p_1,p_2 < \infty$ and 
$1/p = 1/p_1+1/p_2$. 
In particular, 
the critical $\alpha(p_1,p_2, p)$  
is obtained in the case $p> 1$; see 
\cite[Theorem 1.2]{BGHH}.

Now, the purpose of the present paper is 
to generalize and improve the results 
mentioned above. 
We shall consider 
symbols satisfying 
\eqref{eqcondition-for-BSW00} 
and prove estimate 
of the forms 
\eqref{eq221-4-epsilon}, 
\eqref{eq221-calB}, and 
\eqref{eq221-4infty} 
with 
$L^{2}\times L^{2} \to L^{1}$ replaced by 
$L^{p_1}\times L^{p_2} \to L^p$ 
for all $p_1, p_2, p$ 
satisfying $1\le p_1, p_2 \le \infty$ and 
$0\le 1/p\le 1/p_1 + 1/p_2$. 
In particular, our result includes  
\eqref{eqlattice-bump-alpha} 
with critical $\alpha (p_1, p_2, p)$ for $p\le 1$. 
It should also be remarked that our results generalize and 
improve some results of 
\cite{MT-2013}. 
The basic method of the present paper 
is similar to  
\cite{KMT-arxiv} and \cite{KMT-arxiv-2}. 
A new idea is to use 
the Wiener amalgam spaces.

In the following subsections, 
we shall give precise statements of 
our main results 
and make some comments on them.

\subsection{Main results}

Here we shall state the main results of this paper.

We use the following symbol class of $S_{0,0}$-type.

\begin{defn}\label{thBSW00} 
For a nonnegative function 
$W$ on $\R^n \times \R^n$, 
we denote by 
$BS^{W}_{0,0}=BS^{W}_{0,0}(\R^n)$ the set of all those 
smooth functions 
$\sigma = \sigma (x, \xi_1, \xi_2)$ 
on $\R^n \times \R^n \times \R^n$ 
that satisfy the estimate 
\eqref{eqcondition-for-BSW00}
for all multi-indices
$\alpha \in (\N_0)^{3n}$. 
We shall call $W$ the {\it weight function\/} 
of the class $BS^{W}_{0,0}(\R^n)$. 
\end{defn}

For the weight function
$W(\xi_1,\xi_2) = (1+|\xi_1|+|\xi_2|)^{m}$, 
$m\in\R$, 
the class $BS_{0,0}^{W}$ 
is identical with the bilinear H\"ormander symbol class
$BS_{0,0}^{m}$.

We introduce the following. 

\begin{defn}\label{thclassB}
(1) 
Let $0 < q_1,q_2,q \leq \infty$. 
For nonnegative functions $V$ on $\Z^n \times \Z^n$, we define 
\begin{equation*}
\|V\|_{\calB_{q_1,q_2,q}}
=
\sup 
\Bigg\| 
\sum_{ \nu_1,\nu_2 \in \Z^n : \nu_1+ \nu_2=\nu } 
V(\nu_1, \nu_2) A(\nu_1) B(\nu_2) 
\Bigg\|_{\ell^q_{\nu}(\Z^n)}
\end{equation*}
with 
the supremum taken over all 
nonnegative functions
$A \in \ell^{q_1}(\Z^n)$ and $B \in \ell^{q_2}(\Z^n)$ 
with 
$\|A\|_{\ell^{q_1}}=\|B\|_{\ell^{q_2}}=1$. 
We denote by 
$\calB_{q_1,q_2,q}=\calB_{q_1,q_2,q}(\Z^n \times \Z^n)$
the set of 
all nonnegative functions 
$V$ on $\Z^n \times \Z^n$  
such that $\|V\|_{\calB_{q_1,q_2,q}} < \infty$.

(2) 
Let $q_1,q_2,q_3 \in [1, \infty]$.  
For nonnegative functions $V$ on $\Z^n \times \Z^n$, we define 
\begin{equation*}
\|V\|_{\calBform_{q_1,q_2,q_3}}
= 
\sup 
\bigg\{ \sum_{\nu_1, \nu_2\in \Z^n} 
V(\nu_1, \nu_2) A(\nu_1) B(\nu_2) C(\nu_1 + \nu_2)
\bigg\} 
\end{equation*}
with the supremum taken over all 
nonnegative functions 
$A, B, C$ on $\Z^n$ such that 
$\|A\|_{\ell^{q_1}}=\|B\|_{\ell^{q_2}}=\|C\|_{\ell^{q_3}}=1$. 
We denote by 
$\calBform_{q_1,q_2,q_3} =
\calBform_{q_1,q_2,q_3}(\Z^n \times \Z^n)$ 
the set of all nonnegative 
functions $V$ on $\Z^n \times \Z^n$ such that 
$\|V\|_{\calBform_{q_1,q_2,q_3}}< \infty$. 
\end{defn}

If $q_1, q_2, q \in [1, \infty]$, then 
by duality it follows that 
$\calB_{q_1,q_2,q} = \calBform_{q_1, q_2, q^{\prime}}$, 
where $q^{\prime}$ is the conjugate index, $1/q + 1/q^{\prime}=1$. 
The class $\calBform_{2,2,2}$ was already 
introduced in \cite[Definition 1.2]{KMT-arxiv}.

We use the following notation: 
for nonnegative bounded functions $V$  
on $\Z^n \times \Z^n$, 
we define 
\begin{equation*}
\widetilde{V} (\xi_1, \xi_2)
=
\sum_{\nu_1, \nu_2 \in \Z^n} 
V(\nu_1, \nu_2) 
\ichi_{Q} (\xi_1 - \nu_1) 
\ichi_{Q} (\xi_2 - \nu_2), 
\quad 
(\xi_1, \xi_2) \in \R^n \times \R^n,  
\end{equation*}
where $Q=[-1/2, 1/2)^n$. 


Now, Theorem 
\ref{thmain-thm-WA} below 
is the basis of our main results. 
This theorem considers boundedness 
of bilinear pseudo-differential operators 
in Wiener amalgam spaces $W^{p,q}$; 
the definition of Wiener amalgam spaces will be given 
in Subsection \ref{secamalgam}.

\begin{thm}\label{thmain-thm-WA} 
Let $p_1,p_2,p, q_1,q_2,q \in (0,\infty]$ 
and let 
$V$ be a nonnegative bounded function on $\Z^n\times \Z^n$ 
that is not identically equal to $0$. 
Then the relation 
\[
\mathrm{Op}
(BS^{\widetilde{V}}_{0,0} (\R^n))
\subset 
B 
(W^{p_1,q_1} \times W^{p_2,q_2} \to W^{p,q})
\]
holds if and only if 
$1/p \le 1/p_1 + 1/p_2$ and 
$V \in \calB_{q_1,q_2,q}(\Z^n \times \Z^n)$. 
\end{thm}

It is known that 
the Wiener amalgam spaces $W^{p,2}$, $0 < p \le \infty$,
are identical with the $L^2$-based amalgam spaces.
The results on 
pseudo-differential operators in 
$L^2$-based amalgam spaces 
given in \cite[Theorem 1.3]{KMT-arxiv-2} 
(when restricted to bilinear operators)  
are the special case $q_1=q_2=q=2$ 
of Theorem \ref{thmain-thm-WA}.

Several embedding relations 
between the Wiener amalgam spaces and 
the Lebesgue spaces, the local Hardy spaces, and 
$bmo$ are known (see Subsection  \ref{secamalgam}).  
By combining those embedding relations 
with Theorem \ref{thmain-thm-WA},  
we obtain 
the boundedness criterion 
for the latter spaces, 
which reads as follows.

\begin{thm}\label{thLp-calB} 
Let $p_1,p_2 \in [1,\infty]$ and 
$p \in (0,\infty]$ satisfy 
$1/p \le 1/p_1 + 1/p_2$. 
Then the relation 
\begin{equation*}
\op (BS^{\widetilde{V}}_{0,0} (\R^n)) 
\subset B (L^{p_1} \times L^{p_2} \to h^{p}) 
\end{equation*} 
holds for every 
$V \in \calBform_{
\max (2, p_1^{\prime}), \max (2, p_2^{\prime}), \max (2,p)}
(\Z^n \times \Z^n)$, 
where $L^{p_1}$ (resp.\ $L^{p_2}$) can 
be replaced by $bmo$ 
if $p_1=\infty$ (resp.\ $p_2 = \infty$). 
\end{thm}

In the particular case  
$1/p = 1/p_1 + 1/p_2$,  
the class $\calBform$ of Theorem \ref{thLp-calB}  
is assigned as follows.

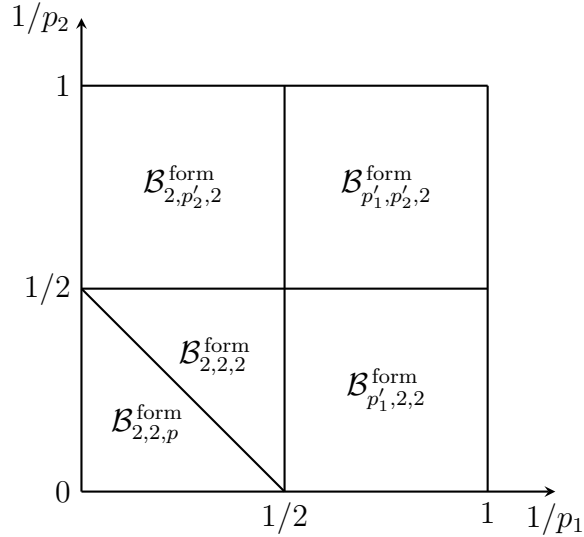
\begin{center}
\begin{tikzpicture}[scale=0.9]
   \draw [thick, -stealth](-0,0)--(7,0) node [anchor=north]{$1/p_1$};
   \draw [thick, -stealth](0,-0)--(0,7) node [anchor=east]{$1/p_2$};
   \node [anchor=east] at (0,0) {0};
   \draw [thick](3,0)--(3,6);
   \draw [thick](0,3)--(6,3);
   \draw [thick](6,0)--(6,6);
   \draw [thick](0,6)--(6,6);
   \draw [thick](0,3)--(3,0);
   \node [anchor=north] at (3,0) {1/2};
   \node [anchor=east] at (0,3) {1/2};
   \node [anchor=north] at (6,0) {1};
   \node [anchor=east] at (0,6) {1};
   \node [font=\normalsize] at (2,2) {$\calBform_{2,2,2}$};
   \node [font=\normalsize] at (1,1) {$\calBform_{2,2,p}$};
   \node [font=\normalsize] at (4.5,1.5) {$\calBform_{p_1',2,2}$};
   \node [font=\normalsize] at (1.5,4.5) {$\calBform_{2,p_2',2}$};
   \node [font=\normalsize] at (4.5,4.5) {$\calBform_{p_1',p_2',2}$};
\end{tikzpicture} 
\end{center}

We shall write 
$\ell^q_{+}(\Z^n \times \Z^n)$ 
or 
$\ell^{q, \infty}_{+}(\Z^n \times \Z^n)$ 
to denote the nonnegative 
functions in the classes 
$\ell^q(\Z^n \times \Z^n)$ 
or 
$\ell^{q, \infty}(\Z^n \times \Z^n)$, 
respectively. 
The next proposition claims that 
certain 
$\ell^q_{+}(\Z^n \times \Z^n)$ 
or 
$\ell^{q, \infty}_{+}(\Z^n \times \Z^n)$ 
are included in the class 
$
\calBform_{
\max (2, p_1^{\prime}), \max (2, p_2^{\prime}), \max (2,p)}
$ of Theorem \ref{thLp-calB}.

\begin{prop}\label{thellq-ellqweak-calB}
Let $p_1, p_2 \in [1, \infty]$ and $p\in (0, \infty]$. 
With the notation 
\begin{equation*}
p_3 =
\left\{
\begin{array}{ll}
{\infty} & {\quad\text{if}\quad 0<p\le 1, }\\ 
{p^{\prime}} & {\quad\text{if}\quad 1<p\le \infty,  }
\end{array}
\right. 
\end{equation*}
we define $q=q(p_1, p_2, p)$ as follows: 
\begin{itemize}
\item[{\rm (I)}]  
if $1/p_1, 1/p_2, 1/p_3 \le 1/2$, then 
$q=4$; 

\item[{\rm (II)}] 
if $1/p_i, 1/p_j \le 1/2 \le 1/p_k $ for 
some permutation $(i,j,k)$ of $(1,2,3)$, 
then 
$q=2p_k$; 

\item[{\rm (III-1)}] 
if $1/p_i \le 1/2 \le 1/p_j, 1/p_k $ 
and 
$1/p_j + 1/p_k \le 3/2$ 
for some permutation $(i,j,k)$ of $(1,2,3)$, 
then 
$1/q= (1/2)(1/p_j + 1/p_k - 1/2)$; 

\item[{\rm (III-2)}]  
if $1/p_i \le 1/2 \le 1/p_j, 1/p_k $ 
and 
$1/p_j + 1/p_k \ge 3/2$ 
for some permutation $(i,j,k)$ of $(1,2,3)$, 
then 
$1/q= 1/p_j + 1/p_k - 1$;

\item[{\rm (IV-1)}] 
if $1/2 \le 1/p_1, 1/p_2, 1/p_3$ 
and if 
\begin{equation*}
{1}/{p_1}+ 
{1}/{p_2} 
\le 1+ {1}/{p_3}, 
\quad 
{1}/{p_2}+ 
{1}/{p_3} 
\le 1+ {1}/{p_1}, 
\quad 
{1}/{p_3}+ {1}/{p_1} 
\le 1+ {1}/{p_2},  
\end{equation*}
then 
$1/q= (1/2)(1/p_1 + 1/p_2 + 1/p_3 - 1)$; 

\item[{\rm (IV-2)}]  
if $1/2 \le 1/p_1, 1/p_2, 1/p_3$ 
and if 
${1}/{p_i}+ 
{1}/{p_j} 
\ge 1+ {1}/{p_k}$ 
for some permutation $(i,j,k)$ of $(1,2,3)$, 
then
 $1/q= 1/p_i + 1/p_j  - 1$.  
\end{itemize}
Then the inclusion 
\begin{equation}\label{eq-inclusion-aaa}
\ell^{q}_{+} (\Z^n\times \Z^n) 
\subset 
\calBform_{
\max (2, p_1^{\prime}), \max (2, p_2^{\prime}), \max (2,p)} (\Z^n\times \Z^n) 
\end{equation}
holds. 
Moreover, 
the stronger inclusion 
\begin{equation}\label{eq-inclusion-bbb}
\ell^{q,\infty}_{+} (\Z^n\times \Z^n) 
\subset 
\calBform_{
\max (2, p_1^{\prime}), \max (2, p_2^{\prime}), \max (2,p)}
 (\Z^n\times \Z^n) 
\end{equation}
holds 
in the following cases: 
\begin{itemize}
\item[{\rm (I)} ] 
$1/p_1, 1/p_2, 1/p_3 \le 1/2$; 
\item[{\rm (II$\ast$)}] 
$1/p_i, 1/p_j \le 1/2 \le 1/p_k < 1$ for 
some permutation $(i,j,k)$ of $(1,2,3)$;  
\item[{\rm (III-1$\ast$)}] 
$1/p_i \le 1/2 \le 1/p_j, 1/p_k $  
and 
$1/p_j + 1/p_k < 3/2$ 
for some permutation $(i,j,k)$ of $(1,2,3)$;

\item[{\rm (IV-1$\ast$)}] 
$1/2 \le 1/p_1, 1/p_2, 1/p_3$, 
\begin{equation*}
{1}/{p_1}+ 
{1}/{p_2} 
< 1+ {1}/{p_3}, 
\quad 
{1}/{p_2}+ 
{1}/{p_3} 
< 1+ {1}/{p_1}, 
\quad 
{1}/{p_3}+ {1}/{p_1} 
< 1 + {1}/{p_2},  
\end{equation*}
and $1/p_1 + 1/p_2 + 1/p_3 \le 2$. 
\end{itemize}
\end{prop}

Notice that $1\le q(p_1, p_2, p) \le 4$ in all cases 
and that $2\le q(p_1, p_2, p) \le 4$ in the cases 
\rm{(I)}, \rm{(II$\ast$)}, 
\rm{(III-1$\ast$)}, and \rm{(IV-1$\ast$)}. 
For a simpler presentation of 
$q(p_1, p_2, p)$ for the case $1/p=1/p_1+1/p_2$, 
see Remark \ref{thHoldercase} below.

Using Theorem \ref{thLp-calB} and 
Proposition \ref{thellq-ellqweak-calB}, 
we can 
prove the theorem below, 
which gives generalizations of 
the estimate \eqref{eq221-4infty} 
to the case $L^{p_1}\times L^{p_2} \to L^{p}$ 
with general $p_1, p_2, p$.

\begin{thm}\label{thLp-diff-Lq}
Suppose $p_1,p_2 \in [1,\infty]$ and 
$p \in (0,\infty]$ satisfy  
$1/p \le 1/p_1 + 1/p_2$. 
Let $q=q(p_1, p_2, p)$ be the number given in 
Proposition \ref{thellq-ellqweak-calB}. 
Then 
there exist a positive constant $c$ and a positive integer $K$ 
depending only on $n, p_1, p_2$, and $p$ such that 
the inequality 
\begin{equation}\label{eqLq-estimate-of-operator-norm}
\|T_{\sigma}\|_{L^{p_1} \times L^{p_2} \to h^{p} } 
\le 
c 
\sum_{|\alpha|\le K} 
\left\|
\sup_{x\in \R^n} \left| 
\partial_{x,\xi_1, \xi_2}^{\alpha}
\sigma (x, \xi_1, \xi_2) 
\right| 
\right\|_{L^{q}_{\xi_1, \xi_2} (\R^n \times \R^n)}
\end{equation} 
holds for all $\sigma \in C^{\infty}((\R^n)^3)$, 
where  
$L^{p_1}$ (resp.\ $L^{p_2}$) can be replaced by $bmo$ 
if $p_1=\infty$ (resp.\ $p_2 = \infty$).  
Moreover, in the cases {\rm (I)}, {\rm (II$\ast$)}, 
{\rm (III-1$\ast$)}, and {\rm (IV-1$\ast$)}, 
the estimate \eqref{eqLq-estimate-of-operator-norm},   
with the same replacement of 
$L^{p_1}$ or $L^{p_2}$ by $bmo$, 
holds with the  
$L^{q}_{\xi_1, \xi_2} (\R^n \times \R^n)$-norm 
on the right hand side replaced 
by the 
$L^{q, \infty}_{\xi_1, \xi_2} (\R^n \times \R^n)$-norm. 
\end{thm}

Theorems \ref{thLp-calB} and \ref{thLp-diff-Lq} 
generalize and improve 
the result of \cite{MT-2013};  
see Subsection \ref{comments}.

As an application of Theorem \ref{thLp-diff-Lq}, 
we have the following corollary. 
This corollary 
is a generalization of \eqref{eq221-4-epsilon} 
to the case $L^{p_1}\times L^{p_2} \to L^{p}$  
with general $p_1, p_2, p$.

\begin{cor}\label{thLp-Ltildeq-GN}
Let $p_1,p_2\in [1, \infty]$ and $p\in (0, \infty]$ satisfy 
$1/p\le 1/p_1 + 1/p_2$ and let 
$q=q(p_1, p_2, p)$ be the number given in 
Proposition \ref{thellq-ellqweak-calB}. 
Assume $\sigma \in C^\infty((\R^n)^3)$
satisfies 
\begin{equation*}
m(\xi_1, \xi_2)=
\sup_{x\in \R^n}|\sigma (x, \xi_1, \xi_2)| 
\in L^{\widetilde{q}}_{\xi_1, \xi_2} (\R^{2n}) 
\end{equation*} 
for some $\widetilde{q}$ satisfying $0< \widetilde{q} < q$. 
Also assume 
\begin{equation*}
\sup_{|\alpha|\le N} 
\left\| \partial^\alpha 
\sigma \right\|_{L^\infty((\R^n)^3)}
\leq A
\end{equation*}
for a sufficiently large positive integer $N$ which 
depends only on $n, p_1, p_2, p$, and $\widetilde{q}$. 
Then the bilinear pseudo-differential operator
$T_{\sigma}$ satisfies
\[
\| T_{\sigma} \|_{L^{p_1} \times L^{p_2} \to h^p} 
\le 
c 
A^{1-\widetilde{q}/q}
\| m \|_{L^{\widetilde{q}}(\R^{2n})}^{\widetilde{q}/q}, 
\]
where 
$L^{p_1}$ (resp.\ $L^{p_2}$) can be replaced by $bmo$ 
if $p_1=\infty$ (resp.\ $p_2 = \infty$) 
and $c$ is a constant depending only 
on $n, p_1, p_2, p$, and $\widetilde{q}$. 
\end{cor}

\begin{rem}\label{thHoldercase}
In the particular case  
$1/p = 1/p_1 + 1/p_2$, 
the number $q=q(p_1, p_2, p)$ of 
Proposition \ref{thellq-ellqweak-calB} and 
Theorem \ref{thLp-diff-Lq} is given as follows: 
\begin{enumerate}
\renewcommand{\labelenumi}{(\theenumi)}
\item
if $1/p_1, 1/p_2\le 1/2 \le 1/p_1 + 1/p_2$, 
then $q=4$; 
\item
if $1/p_1, 1/p_2, 1/p_1 + 1/p_2 \le 1/2 $, 
then 
$1/q = (1/2)(1- 1/p_1 - 1/p_2)$; 
\item
if $1/p_1\le 1/ 2 \le 1/ p_2 $, then 
$q=2p_2$; 
\item
if $1/p_2\le 1/ 2 \le 1/ p_1 $, then 
$q=2p_1$; 
\item
if $1/ 2 \le 1/ p_1, 1/p_2 $, and 
$1/p_1+ 1/p_2 \le 3/2$, 
then 
${1}/{q}=({1}/{2}) ( {1}/{p_1}+ {1}/{p_2} -{1}/{2})$; 
\item 
if $1/ 2 \le 1/ p_1, 1/p_2 $, and 
$1/p_1+ 1/p_2 \ge 3/2$, 
then 
${1}/{q}={1}/{p_1}+ {1}/{p_2} - 1$. 
\end{enumerate}
Moreover, the cases 
(I), (II$\ast$), and (III-1$\ast$) of 
Proposition \ref{thellq-ellqweak-calB} 
correspond to the case 
$0\le  1/p_1, 1/p_2 <1$ and $0< 1/p= 1/p_1 + 1/p_2< 3/2$. 

The number $q=q(p_1,p_2,p)$ for  
$1/p = 1/p_1+1/p_2$ is assigned as follows. 

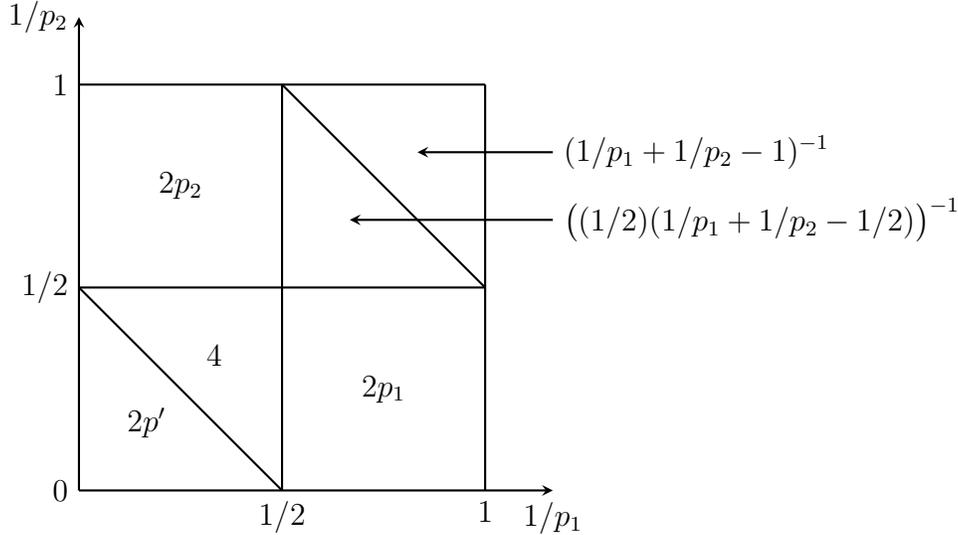
\begin{center}
\begin{tikzpicture}[scale=0.9]
   \draw [thick, -stealth](-0,0)--(7,0) node [anchor=north]{$1/p_1$};
   \draw [thick, -stealth](0,-0)--(0,7) node [anchor=east]{$1/p_2$};
   \node [anchor=east] at (0,0) {0};
   \draw [thick](3,0)--(3,6);
   \draw [thick](0,3)--(6,3);
   \draw [thick](6,0)--(6,6);
   \draw [thick](0,6)--(6,6);
   \draw [thick](0,3)--(3,0);
   \draw [thick](3,6)--(6,3);
   \node [anchor=north] at (3,0) {1/2};
   \node [anchor=east] at (0,3) {1/2};
   \node [anchor=north] at (6,0) {1};
   \node [anchor=east] at (0,6) {1};
   \node [font=\normalsize] at (2,2) {$4$};
   \node [font=\normalsize] at (1,1) {$2p'$};
   \node [font=\normalsize] at (4.5,1.5) {$2p_1$};
   \node [font=\normalsize] at (1.5,4.5) {$2p_2$};
   \draw [thick, -stealth](7,4)--(4,4) ;
   \node [font=\normalsize, anchor=west] at (7,4) 
   {$\big(({1}/{2})({1}/{p_1}+{1}/{p_2}-{1}/{2}) \big)^{-1}$};
   \draw [thick, -stealth](7,5)--(5,5) ;
   \node [font=\normalsize, anchor=west] at (7,5) {$({1}/{p_1}+{1}/{p_2}-1)^{-1}$};
\end{tikzpicture} 
\end{center}

We shall see that the number 
$q=q(p_1, p_2, p)$ for 
$1/p = 1/p_1+1/p_2$ 
is sharp (Subsection \ref{sharpness-q}). 
\end{rem}

\begin{rem}\label{thremark-lattice-bump}
Here is a result that suggests close relation 
between the Wiener amalgam space  
and 
the norm $\|\cdot \|_{\calB_{q_1,q_2,q}}$ 
through certain bilinear Fourier multiplier operators.  
Let $\Phi$ be a function in 
$C_{0}^{\infty}(\R^{2n})$ such that 
$\supp \Phi \subset 
[-1+\epsilon , 1-\epsilon ]^{2n}$ 
for some $\epsilon \in (0, 1)$ and 
$\Phi (0) \neq 0$. 
For $M\in \ell^{\infty}(\Z^n \times \Z^n)$, 
let 
\[
\sigma_{M}(\xi_1, \xi_2) 
=
\sum_{\nu_1, \nu_2 \in \Z^n} 
M(\nu_1, \nu_2) 
\Phi 
(\xi_1- \nu_1, \xi_2 - \nu_2), 
\quad \xi_1, \xi_2 \in \R^n, 
\]
and let 
$T_{\sigma_M}$ be the corresponding 
bilinear Fourier multiplier operator.  
We extend the definition of 
$\|\cdot \|_{\calB_{q_1,q_2,q}}$ 
to complex valued functions;  
for complex-valued 
$M\in \ell^{\infty}(\Z^n \times \Z^n)$
and for 
$q_1, q_2, q \in (0, \infty]$, 
we define 
\begin{equation*}
\|M\|_{\calB_{q_1,q_2,q}}
=
\sup 
\Bigg\| 
\sum_{ \nu_1,\nu_2 \in \Z^n : \nu_1+ \nu_2=\nu } 
M(\nu_1, \nu_2) A(\nu_1) B(\nu_2) 
\Bigg\|_{\ell^q_{\nu}(\Z^n)},
\end{equation*}
where 
the supremum is taken over all 
finitely supported 
$A :\Z^n \to \C$ and $B: \Z^n\to \C$ 
with 
$\|A\|_{\ell^{q_1}}=\|B\|_{\ell^{q_2}}=1$. 
Then,  
by modifying the method of the present paper, 
it is possible to prove the following:  
{\it if  
$p_1, p_2, p, q_1, q_2, q \in (0,\infty]$ 
and $1/p \le 1/p_1+1/p_2$, then  
there exists a constant $c\in (0,\infty)$ 
such that 
the inequalities 
\[
c^{-1}
\|M\|_{\calB_{q_1, q_2, q}}
\le 
\|T_{\sigma_M}\|_{W^{p_1, q_1}\times W^{p_2, q_2} \to W^{p,q}}
\le 
c
\|M\|_{\calB_{q_1, q_2, q}}
\]
hold for all 
complex-valued 
$M \in \ell^{\infty}(\Z^n \times \Z^n)$.}
Details will be discussed elsewhere.  
\end{rem}

\subsection{Some comments on the main results}
\label{comments}

We shall observe that Theorems  
\ref{thLp-calB} and \ref{thLp-diff-Lq} give 
generalizations and refinements 
of the estimate \eqref{eqbddp1p2p}--\eqref{eqcriticalorder} 
of \cite{MT-2013}. 
Here we only consider the case 
$1/p=1/p_1+1/p_2$.

We write 
\begin{equation}\label{eqcriticalweightforHormander}
V_{\ast}(\nu_1, \nu_2)
=
(1+|\nu_1|+|\nu_2|)^{m(p_1, p_2, p)}, 
\quad 
\nu_1, \nu_2 \in \Z^n, 
\end{equation}
where 
$m(p_1, p_2, p)$ is given by \eqref{eqcriticalorder}. 
Recall that 
\[
\calBform_{
\max (2, p_1^{\prime}), \max (2, p_2^{\prime}), \max (2,p)}
 (\Z^n\times \Z^n) 
\] 
is the class given in 
Theorem \ref{thLp-calB}. 
Recall also the number $q(p_1, p_2, p)$ 
mentioned in Remark \ref{thHoldercase}.

Now, if 
$0 \le 1/{p_1}, 1/{p_2} <1$ and 
$0< {1}/{p}={1}/{p_1}+ {1}/{p_2}< {3}/{2}$, 
then $m(p_1, p_2, p)=-2n/q(p_1, p_2, p)$ and 
\[
V_{\ast}
\in \ell^{q(p_1, p_2, p), \infty}_{+} (\Z^n \times \Z^n) 
\subset 
\calBform_{
\max (2, p_1^{\prime}), \max (2, p_2^{\prime}), \max (2,p)}
 (\Z^n\times \Z^n) 
\]
(see Proposition \ref{thellq-ellqweak-calB} and 
Remark \ref{thHoldercase}). 
Moreover, in Section \ref{sectionWeightclass} 
we shall show that 
the relation 
\[
V_{\ast} 
\in \calBform_{
\max (2, p_1^{\prime}), \max (2, p_2^{\prime}), \max (2,p)}
 (\Z^n\times \Z^n) 
\]
still holds 
for $p_1, p_2$ such that 
\begin{align*}
&
0\le 1/p_1 \le 1/2, 
\;\; 1/p_2=1, 
\\
&
\text{or}\quad 
0\le 1/p_2 \le 1/2, 
\;\;
1/p_1 = 1, 
\\
&
\text{or}\quad 
1/p_1=1/p_2=0, 
\\
&
\text{or}\quad 
1/2 < 1/p_1 , 1/p_2 <1,
\;\; 1/p_1 + 1/p_2 =3/2  
\end{align*}
(see Proposition \ref{th-22infty} (2) 
and 
Proposition \ref{th-p1p2q-primeprod} (2)). 
These facts mean that 
Theorems \ref{thLp-calB} and \ref{thLp-diff-Lq} 
give generalizations 
of the estimate \eqref{eqbddp1p2p}--\eqref{eqcriticalorder} 
of \cite[Theorem 1.1]{MT-2013}.

Moreover, 
Theorems \ref{thLp-calB} and \ref{thLp-diff-Lq} 
give estimates that are stronger 
than \eqref{eqbddp1p2p}. 
For example, 
Theorems \ref{thLp-calB} and \ref{thLp-diff-Lq} 
give estimates 
for $L^1\times L^{p_2} \to h^{p}$, 
$2\le p_2 \le \infty$, or 
$bmo \times bmo \to L^{\infty}$ 
in the cases where 
only weaker estimates 
$h^1\times L^{p_2} \to h^{p}$ or   
$bmo \times bmo \to bmo$ 
are claimed in \eqref{eqbddp1p2p}.

In the case $1/2 < 1/p_1,  1/p_2 \le 1$ and 
${1}/{p}={1}/{p_1}+ {1}/{p_2}> 3/2$, 
the weight function 
\eqref{eqcriticalweightforHormander} 
does not 
belong to the class 
$\calBform_{p_1^{\prime}, p_2^{\prime}, 2}$ 
of Theorem \ref{thLp-calB} 
(see Proposition \ref{thcounterexBp1p22}) 
while the slightly smaller weight 
$(1 + |\nu_1|+|\nu_2|)^{m}$ 
with $m < m(p_1,p_2,p)$ belongs to 
$\calBform_{p_1^{\prime}, p_2^{\prime}, 2}$ 
(see Proposition \ref{th-p1p2q-primeprod} (1)). 
\vs

We end this section with noting the plan of this paper.
In Section \ref{section2}, 
we will give the basic notations used throughout the paper 
and recall the definitions and properties of some function spaces.
In Section \ref{sectionWeightclass}, 
we give several properties of the classes  
$\calB_{q_1,q_2,q}$ 
and $\calBform_{q_1,q_2,q_3}$ and 
prove Proposition \ref{thellq-ellqweak-calB}. 
In Section \ref{sectionProof}, 
we prove the main theorems of the paper. 
In Appendix, we give a proof of a 
Gagliardo-Nirenberg type inequality 
that is used to derive 
Corollary \ref{thLp-Ltildeq-GN} 
from Theorem \ref{thLp-diff-Lq}.

\section{Preliminaries}\label{section2}
\subsection{Basic notations} 

Here we collect notations which will be used throughout this paper.

We denote by $\N$ and $\N_0$
the sets of positive integers 
and nonnegative integers, respectively. 
We denote by $Q$ the $n$-dimensional 
unit cube $[-1/2,1/2)^n$. 
For $x\in \R^d$, we write 
$\langle x \rangle = (1 + | x |^2)^{1/2}$. 
Thus $\langle (x,y) \rangle = (1+|x|^2+|y|^2)^{1/2}$ 
for $(x,y) \in \R^n \times \R^n$.

For two nonnegative functions $A(x)$ and $B(x)$ defined 
on a set $X$, 
we write $A(x) \lesssim B(x)$ for $x\in X$ to mean that 
there exists a positive constant $C$ such that 
$A(x) \le CB(x)$ for all $x\in X$. 
We often omit to mention the set $X$ when it is 
obviously recognized.  
Also $A(x) \approx B(x)$ means that
$A(x) \lesssim B(x)$ and $B(x) \lesssim A(x)$.

We denote the Schwartz space of rapidly 
decreasing smooth functions
on $\R^d$ 
by $\calS (\R^d)$ 
and denote its dual,
the space of tempered distributions, 
by $\calS^\prime(\R^d)$. 
The Fourier transform and the inverse 
Fourier transform of $f \in \calS(\R^d)$ are given by
\begin{align*}
\mathcal{F} f  (\xi) 
&= \widehat {f} (\xi) 
= \int_{\R^d}  e^{-i \xi \cdot x } f(x) \, dx, 
\\
\mathcal{F}^{-1} f (x) 
&= \check f (x)
= \frac{1}{(2\pi)^d} \int_{\R^d}  e^{i x \cdot \xi } f( \xi ) \, d\xi,
\end{align*}
respectively. 
For $m \in \calS^\prime (\R^d)$, 
the linear Fourier multiplier operator is defined by
\begin{equation*}
m(D) f 
=
\mathcal{F}^{-1} \left[m \cdot \mathcal{F} f \right].
\end{equation*}

For a measurable subset $E \subset \R^d$, 
the Lebesgue space $L^p (E)$, $0<p\le \infty$, 
is the set of all those 
measurable functions $f$ on $E$ such that 
$\| f \|_{L^p(E)} = 
\left( \int_{E} \big| f(x) \big|^p \, dx \right)^{1/p} 
< \infty 
$
if $0< p < \infty$ 
or   
$\| f \|_{L^\infty (E)} 
= 
\mathrm{ess}\, \sup_{x\in E} |f(x)|
< \infty$ if $p = \infty$. 
We also use the notation 
$\| f \|_{L^p(E)} = \| f(x) \|_{L^p_{x}(E)} $ 
when we want to indicate the variable explicitly.

Let $\K$ be a countable set. 
The sequence spaces 
$\ell^q (\K)$ and $\ell^{q, \infty} (\K)$ 
are defined as follows.  
The space $\ell^q (\K)$, $0 < q \le \infty$, 
consists of all those 
complex sequences $a=\{a_k\}_{k\in \K}$ 
such that 
$ \| a \|_{ \ell^q (\K)} 
= 
\left( \sum_{ k \in \K } 
| a_k |^q \right)^{ 1/q } <\infty$ 
if $0 < q < \infty$  
or 
$\| a \|_{ \ell^\infty (\K)} 
= \sup_{k \in \K} |a_k| < \infty$ 
if $q = \infty$.
For $0< q<\infty$, 
the space 
$\ell^{ q,\infty }(\K)$ is 
the set of all those complex 
sequences 
$a= \{ a_k \}_{ k \in \K }$ such that
\begin{equation*}
\| a \|_{\ell^{q,\infty} (\K) } 
=\sup_{t>0} 
\big\{ t \big( 
\card  \{ k \in \K : | a_k | > t \} 
\big)^{1/q} \big\} < \infty. 
\end{equation*}
Sometimes we write 
$\| a \|_{\ell^{q}} = 
\| a_k \|_{\ell^{q}_k }$ or 
$\| a \|_{\ell^{q,\infty}} = 
\| a_k \|_{\ell^{q,\infty}_k }$.  

We end this subsection by
noting the definition of the class $\calM(\R^d)$,
which was introduced in
\cite[Definition 3.5]{KMT-arxiv-2}
(see also \cite[Definition 3.7]{KMT-arxiv}).

\begin{defn}\label{thdefModerate}
Let $d \in \N$. 
We say that a continuous function 
$F: \R^d \to (0, \infty)$   
is of {\it moderate class\/} if  
there exist constants $C=C_{F}>0$ 
and $M=M_{F}>0$ such that 
\begin{equation*}
F(\xi + \eta ) 
\le C F(\xi)
\langle \eta \rangle ^{M}
\; \; \text{for all} \;\; \xi, \eta \in \R^d.  
\end{equation*}
We denote by 
$\calM (\R^d)$ the set of 
all functions on $\R^d$ of moderate class. 
\end{defn}

\subsection{Local Hardy spaces $h^p$ and the space $bmo$} 
\label{sechardybmo}

We recall the definitions of the local Hardy spaces 
$h^p(\R^n)$ and the space $bmo(\R^n)$.

Let $\phi \in \calS(\R^n)$ be such that
$\int_{\R^n}\phi(x)\, dx \neq 0$. 
Then the local Hardy space $h^p(\R^n)$, $0<p\leq\infty$,  
consists of
all $f \in \calS'(\R^n)$ such that 
$\|f\|_{h^p}=\|\sup_{0<t<1}|\phi_t*f|\|_{L^p}
<\infty$,  
where $\phi_t(x)=t^{-n}\phi(x/t)$.
It is known that $h^p(\R^n)$
does not depend on the choice of the function $\phi$ 
up to the equivalence of quasi-norm. 
If $1 < p \leq \infty$, 
then 
$h^p(\R^n) = L^p(\R^n)$ with equivalent norms. 
If $0 < p \le 1$, then 
the inequality 
$\|f\|_{L^p}\lesssim \|f\|_{h^p}$ 
holds for all 
$f \in h^{p}$ which are 
defined by locally integrable 
functions on $\R^n$.

The space $bmo(\R^n)$ consists of
all locally integrable functions $f$ on $\R^n$ 
such that
\[
\|f\|_{bmo}
=\sup_{|R| \le 1}\frac{1}{|R|}
\int_{R}|f(x)-f_R|\, dx
+\sup_{|R|\geq1}\frac{1}{|R|}
\int_R |f(x)|\, dx
<\infty,
\]
where $f_R=|R|^{-1}\int_R f(x) \, dx$,
and $R$ ranges over all cubes in $\R^n$. 
Obviously 
$L^{\infty} \subset bmo$.

See Goldberg \cite{goldberg 1979} for more details about 
$h^p$ and $bmo$.

\subsection{Wiener amalgam spaces} 
\label{secamalgam}

Let $\kappa\in\calS(\R^n)$ be such that $\supp \kappa$ is compact and
\begin{equation}\label{eqcondofkappa}
\bigg| \sum _{k\in \mathbb{Z}^{n}} \kappa(\xi-k) \bigg| 
\geq 1, 
\quad 
\xi \in \R^n .
\end{equation}
Then for $0 < p,q \leq \infty$,
the Wiener amalgam space $W^{p,q}=W^{p,q}(\R^n)$
is defined to be the set of
all tempered distributions $f \in \calS'(\R^n)$ such that
\begin{equation}\label{eqdefinition-WAnorm}
\|f\|_{W^{p,q}} = 
\Big\| \big\|
\kappa(D-k)f(x)
\big\|_{\ell^{q}_{k}(\Z^n)} 
\Big\|_{L^p_{x}(\R^n)}
\end{equation}
is finite. 
It is known that the definition of Wiener amalgam space is 
independent of the choice of the function $\kappa$ 
up to the equivalence of the quasi-norm 
\eqref{eqdefinition-WAnorm}.
Let us remark that
the usual definition of the Wiener amalgam space 
is given by 
$\kappa$ satisfying 
$\sum _{k\in \mathbb{Z}^{n}} \kappa(\cdot-k) \equiv 1$.
However, for technical reasons,
we shall use a bit more general $\kappa$ 
satisfying \eqref{eqcondofkappa}.

We note that $W^{p,q}$ is a quasi-Banach space 
(Banach space for $1 \leq  p,q \leq \infty$) and 
$\mathcal{S} \subset W^{p,q} \subset \mathcal{S}^\prime$. 
If $0 < p,q < \infty$, then $\mathcal{S}$ is dense in $W^{p,q}$.
See Feichtinger \cite{Fei-1981} and Triebel \cite{Tri-1983}
for more details. 
See also \cite{Fei-1983, GS-2004, Gro-book-2001, Kob-2006, WH-2007} 
for modulation spaces
which are variation of the Wiener amalgam spaces.

The following embedding results are proved in 
\cite{CKS-2015, GCFZ-2019, GWYZ-2017}.

\begin{prop}\label{thWaembd}
\begin{enumerate}\setlength{\itemsep}{3pt} 
\item
$W^{p_1,q_1} \hookrightarrow W^{p_2,q_2}$
if
$0<p_1\le p_2 \le \infty$ and $0<q_1\le q_2 \le \infty$.
\item 
If $1 \le p \le \infty$, then 
$W^{p,\min (2, p')} \hookrightarrow L^{p} 
\hookrightarrow 
W^{p,\max (2, p')} $. 
\item
If $0< p \le 2$, then 
$W^{p,2} \hookrightarrow h^{p}$. 
\item 
$bmo \hookrightarrow W^{\infty,2}$. 
\end{enumerate}
\end{prop}

\section{Classes $\calB_{q_1,q_2,q}$ and $\calBform_{q_1,q_2,q_3}$}
\label{sectionWeightclass}

\subsection{Some basic properties
of the classes $\calB_{q_1,q_2,q}$ and $\calBform_{q_1,q_2,q_3}$}

\begin{prop}\label{thembedBclass}
If 
$q_1,q_2,q , \widetilde{q_1} , \widetilde{q_2} , \widetilde{q} \in (0, \infty]$, 
$q_1\ge \widetilde{q_1}$, 
$q_2 \ge \widetilde{q_2}$, 
and
$q \le \widetilde{q}$, then 
$
\calB_{q_1,q_2,q} \subset 
\calB_{\widetilde{q_1} , \widetilde{q_2} , \widetilde{q}}.  
$
If $q_1,q_2,q_3 , 
\widetilde{q_1} , \widetilde{q_2} , \widetilde{q_3} 
\in [1, \infty]$, 
$q_1\ge \widetilde{q_1}$, 
$q_2 \ge \widetilde{q_2}$, 
and
$q_3 \ge \widetilde{q_3}$, then 
$
\calBform_{q_1,q_2,q_3} \subset 
\calBform_{\widetilde{q_1} , \widetilde{q_2} , \widetilde{q_3}}. 
$
\end{prop}

\begin{proof}
The embeddings follow from the fact
$\ell^{p} \hookrightarrow \ell^q $ for $p \le q$.
\end{proof}

\begin{prop}\label{thVast}
Let $0 < q_1,q_2,q \leq \infty$.
For any $V\in \calB_{q_1,q_2,q}(\Z^n \times \Z^n)$, 
there exists a function 
$V^{\ast}$ in the moderate class $\calM (\R^n \times \R^n)$ 
such that 
$V(\nu_1, \nu_2) \le V^{\ast} (\nu_1, \nu_2)$ 
for all $(\nu_1, \nu_2) \in \Z^n \times \Z^n$ and 
the restriction of $V^{\ast}$ to $\Z^n \times \Z^n$ 
belongs to $\calB_{q_1,q_2,q}(\Z^n \times \Z^n)$. 
\end{prop}

\begin{proof} 
The following argument is a generalization of that given in 
\cite[Proof of Proposition 3.9]{KMT-arxiv}. 
Suppose $V\in \calB_{q_1,q_2,q}(\Z^n \times \Z^n)$. 
We may assume $V \not\equiv 0$.
By the translation invariance of $\ell^q(\Z^n)$, we have
\begin{align}\label{eqBLqtranslation}
\begin{split}
&\Bigg\|
\sum_{ \nu_1,\nu_2 \in \Z^n : \nu_1+ \nu_2=\nu } 
V(\nu_1-\mu_1, \nu_2-\mu_2) 
A(\nu_1) 
B(\nu_2)
\Bigg\|_{\ell^q_{\nu}(\Z^n)}
\\&=
\Bigg\|
\sum_{ \nu_1,\nu_2 \in \Z^n : \nu_1+ \nu_2=\nu -\mu_1 - \mu_2} 
V(\nu_1, \nu_2) 
A(\nu_1+\mu_1) 
B(\nu_2 +\mu_2)
\Bigg\|_{\ell^{q}_{\nu}(\Z^n) }
\le c
\|A\|_{\ell^{q_1} (\Z^n) } 
\|B\|_{\ell^{q_2} (\Z^n) } 
\end{split}
\end{align}
for all $(\mu_1, \mu_2) \in 
\Z^n \times \Z^n$,
where 
$c=\|V\|_{\calB_{q_1,q_2,q}}$. 
Take $M> 2n/\min(1,q)$
and set
\begin{align*}
G(\nu_1, \nu_2) 
&
= \sum_{\mu_1, \mu_2 \in \Z^n} 
V(\nu_1 - \mu_1, \nu_2 - \mu_2) 
\langle  (\mu_1, \mu_2)  \rangle ^{-M} 
\\
&= \sum_{\mu_1, \mu_2 \in \Z^n} 
V(\mu_1, \mu_2) 
\langle  ( \nu_1-\mu_1, \nu_2 - \mu_2)  
\rangle ^{-M} .
\end{align*}
Then this $G$ belongs to 
$\calB_{q_1,q_2,q}(\Z^n \times \Z^n)$.
In fact, 
since
\begin{align*}
&
\Bigg\|
\sum_{ \nu_1,\nu_2 \in \Z^n : \nu_1+ \nu_2=\nu } 
G(\nu_1, \nu_2) 
A(\nu_1) 
B(\nu_2)
\Bigg\|_{\ell^{q}_{\nu}(\Z^n) }
^{\min(1,q)}
\\&\leq
\sum_{\mu_1, \mu_2 \in \Z^n} 
\langle  (\mu_1, \mu_2)  \rangle^{-M\min(1,q)} 
\Bigg\|
\sum_{ \nu_1,\nu_2 \in \Z^n : \nu_1+ \nu_2=\nu } 
V(\nu_1-\mu_1, \nu_2 - \mu_2) 
A(\nu_1) 
B(\nu_2)
\Bigg\|_{\ell^{q}_{\nu}(\Z^n) }^{\min(1,q)} ,
\end{align*} 
the inequality \eqref{eqBLqtranslation} implies 
$G \in \calB_{q_1,q_2,q}(\Z^n \times \Z^n)$. 
Obviously, $G(\nu_1, \nu_2) 
\ge V(\nu_1, \nu_2)$ since
$V$ is nonnegative.
If we define the function $V^{\ast}$ 
by
\[
V^{\ast} (\xi_1, \xi_2)
=
\sum_{\mu_1, \mu_2 \in \Z^n} 
V(\mu_1, \mu_2) 
\langle  ( \xi_1-\mu_1, \xi_2 - \mu_2)  
\rangle ^{-M}, 
\quad (\xi_1, \xi_2) \in \R^n \times \R^n,   
\]
then $V^\ast$ belongs to the moderate class $\calM(\R^n \times \R^n)$   
and 
has the desired properties. 
\end{proof}

\subsection{
Proof of Proposition \ref{thellq-ellqweak-calB}}

In this subsection, we shall prove 
 Proposition \ref{thellq-ellqweak-calB}. 
The idea of the argument here 
is the same as in 
\cite[Proposition 3.4]{KMT-arxiv}, 
where the inclusion 
$\ell^{4, \infty}_{+} \subset \calBform_{2,2,2}$ was proved. 
We use the following lemma.

\begin{lem}\label{thBrascamp-Lieb}
Let 
$q_i, r_i \in [1, \infty]$, $i=0, 1,2, 3$. 
\\
$(1)$ The inequality 
\begin{equation}\label{eqRnBL}
\begin{split}
&
\int_{\R^n \times \R^n}
V(x_1, x_2)
A(x_1)
B(x_2)
C(x_1+x_2)\, dx_1 dx_2 
\\
&\lesssim
\|V\|_{L^{q_0} (\R^{n} \times \R^n) }
\|A\|_{L^{q_1} (\R^{n}) }
\|B\|_{L^{q_2} (\R^{n}) }
\|C\|_{L^{q_3} (\R^{n}) }
\end{split}
\end{equation}
holds for all nonnegative measurable functions 
$V$ on $\R^n \times \R^n$ and 
all nonnegative measurable functions 
$A, B, C$ on 
$\R^n $ 
if and only if $(q_i)$ satisfy 
\begin{align}
& \label{eqLqeq}
{2}/{q_0}
+ {1}/{q_1}
+ {1}/{q_2}
+ {1}/{q_3} =2, 
\\
& \label{eqLqineq}
0\le {1}/{q_i} \le 1 - 
{1}/{q_0}\leq1, 
\quad i = 1, 2, 3.   
\end{align}

$(2)$ The inequality 
\begin{equation}\label{eqRnBLLorentz}
\begin{split}
&
\int_{\R^n \times \R^n}
V(x_1, x_2)
A(x_1)
B(x_2)
C(x_1+x_2)\, dx_1 dx_2 
\\
&\lesssim
\|V\|_{L^{q_0, r_0} (\R^{n} \times \R^n) }
\|A\|_{L^{q_1, r_1} (\R^{n}) }
\|B\|_{L^{q_2, r_2} (\R^{n}) }
\|C\|_{L^{q_3, r_3} (\R^{n}) }
\end{split}
\end{equation}
holds for all nonnegative measurable functions 
$V$ on $\R^n \times \R^n$ and 
all nonnegative measurable functions 
$A, B, C$ on 
$\R^n $ 
if $(q_i)$ satisfy \eqref{eqLqeq} and
\begin{equation}\label{eqLqineqstrict}
0< {1}/{q_i} < 1 - {1}/{q_0} < 1, 
\quad i = 1, 2, 3,   
\end{equation}
and if $(r_i)$ satisfy
\begin{equation}\label{eqHolder}
{1}/{r_0} +
{1}/{r_1}+ 
{1}/{r_2}+
{1}/{r_3}
\ge 1. 
\end{equation}

\end{lem}

\begin{proof} 
The claim (1) is a special case of 
the Brascamp--Lieb inequalities.  
An elementary proof of the `if' part can be found 
in \cite[Proof of Proposition 3.4]{KMT-arxiv}. 
The claim (2) can be deduced from 
the inequalities of (1) with the aid of real interpolation 
for multilinear operators of 
Janson \cite{Janson} (see 
\cite[Proof of Proposition 3.4]{KMT-arxiv}). 
(In \cite[loc. cit.]{KMT-arxiv}, the condition 
$\sum_{i=0}^{3} r_i =1$ is given in stead of 
\eqref{eqHolder}.  
But the condition \eqref{eqHolder} is also sufficient 
since if we replace $r_i$ by larger numbers then 
we only obtain stronger conclusion by virtue of 
the embedding 
$L^{q_i, r_i}\hookrightarrow L^{q_i, \widetilde{r_i}}$ 
for $r_i \le \widetilde{r_i}$.)  
\end{proof}

\begin{proof}[Proof of Proposition \ref{thellq-ellqweak-calB}]
Suppose 
$p_1, p_2 \in [1,\infty]$ and 
$p \in (0, \infty]$ are given 
and consider the problem to find 
$q\in [1, \infty]$ that satisfies  
\eqref{eq-inclusion-aaa} or \eqref{eq-inclusion-bbb}. 
We shall use the notation $p_3$ given in 
Proposition \ref{thellq-ellqweak-calB}.  
Then $p_3 \in [1, \infty]$ and 
$\max (2, p)= \max (2, p_3^{\prime})$. 
Thus the problem is to find $q$ such that 
\begin{equation*}
{\ell^{q}_{+}} \;\; \text{or}\;\; 
{\ell^{q, \infty}_{+}} 
\subset 
\calBform_{
\max (2, p_1^{\prime}), 
\max (2, p_2^{\prime}), \max (2,p_3^{\prime})}. 
\end{equation*}
Obviously, this inclusion holds 
if there exist $q_1, q_2, q_3 \in [1, \infty]$ such that 
\begin{equation}\label{eqcondition}
{\ell^{q}_{+}} \;\; \text{or}\;\; 
{\ell^{q, \infty}_{+}} 
\subset 
\calBform_{q_1, q_2, q_3} 
\subset 
\calBform_{\max (2, p_1^{\prime}), 
\max (2, p_2^{\prime}), \max (2,p_3^{\prime}) }. 
\end{equation}
By applying the inequalities of 
Lemma 
\ref{thBrascamp-Lieb} 
to appropriate step functions, 
we see that 
the first $\subset$ of \eqref{eqcondition} for $\ell^{q}_{+}$  
holds if 
\begin{align*}
& 
{2}/{q}
+ {1}/{q_1}
+ {1}/{q_2}
+ {1}/{q_3} =2, 
\\
& 
0\le {1}/{q_\ell} \le 
1 - {1}/{q} \le 1, 
\quad 
\ell = 1, 2, 3, 
\end{align*}
and the first $\subset$ of \eqref{eqcondition} for $\ell^{q, \infty}_{+}$  
holds if 
\begin{align*}
& 
{2}/{q}
+ {1}/{q_1}
+ {1}/{q_2}
+ {1}/{q_3} =2, 
\\
& 
0< {1}/{q_\ell} < 1 - {1}/{q} < 1, 
\quad 
\ell = 1, 2, 3, 
\\
&
{1}/{q_1} + {1}/{q_2} + {1}/{q_3} \ge 1. 
\end{align*}
By Proposition \ref{thembedBclass}, 
the second $\subset$ of \eqref{eqcondition} 
holds if 
\[
q_1 \ge \max (2, p_1^{\prime}), 
\quad 
q_2 \ge \max (2, p_2^{\prime}), 
\quad 
q_3 \ge \max (2, p_3^{\prime}).  
\]
By arranging the inequalities, 
we see the following. 
Firstly,  
if $q_1, q_2, q_3$ satisfy 
\begin{align}
\tag{A-1} 
&
{1}/{q_\ell}\le 
1- \max \big\{ {1}/{2}, {1}/{p_\ell} \big\}, 
\quad \ell = 1,2,3, 
\\
&
\tag{A-2}
0\le {1}/{q_1}\le {1}/{q_2}+ {1}/{q_3}, 
\quad 
0\le {1}/{q_2}\le {1}/{q_3}+ {1}/{q_1}, 
\quad 
0\le {1}/{q_3}\le {1}/{q_1}+ {1}/{q_2},
\end{align}
then 
the inclusion 
\eqref{eq-inclusion-aaa} holds with 
\begin{equation}\label{eqqsumqi}
{1}/{q}
=
1- ({1}/{2}) \big( 
{1}/{q_1} + {1}/{q_2} + {1}/{q_3}
\big). 
\end{equation}
Secondly, if 
$q_1, q_2, q_3$ satisfy (A-1) and 
\begin{align}
\tag{A-2$\ast$} 
&
0< {1}/{q_1}<  {1}/{q_2}+ {1}/{q_3}, 
\quad 
0< {1}/{q_2} < {1}/{q_3}+ {1}/{q_1}, 
\quad 
0< {1}/{q_3} < {1}/{q_1}+ {1}/{q_2},
\\
\tag{A-3} 
&
{1}/{q_1} + {1}/{q_2} + {1}/{q_3}\ge 1,
\end{align}   
then the inclusion 
\eqref{eq-inclusion-bbb} holds for 
$q$ of 
\eqref{eqqsumqi}.

We would like to have $q$ as large as possible. 
Thus, for given $p_1, p_2, p_3\in [1, \infty]$, 
we shall find the best choice of 
$q_1, q_2, q_3$, {\it i.e.\/}, 
that satisfy the above conditions (A-1), (A-2) or 
(A-1), (A-2$\ast$), (A-3) 
and maximize (if any)  
$\sum_{\ell=1}^{3}1/q_{\ell}$. 
Since this is an elementary task, 
we shall give only the result below, which will 
complete Proof of Proposition \ref{thellq-ellqweak-calB}.

Hereafter 
$(i, j, k)$ will denotes any permutation of $(1, 2, 3)$. 

{Case (I)} ${1}/{p_1}, {1}/{p_2}, {1}/{p_3} \le {1}/{2}$. 
In this case, the best 
choice of 
$(q_\ell)$ 
is 
${1}/{q_1}={1}/{q_2}={1}/{q_3}={1}/{2}$ and 
the $q$ given by \eqref{eqqsumqi} is  
${1}/{q}= {1}/{4}$. 
In this case, this 
$(q_\ell)$ satisfy 
(A-2$\ast$) and (A-3) as well and hence $q=4$ satisfies 
\eqref{eq-inclusion-bbb}.

{Case (II)}   
${1}/{p_i}, {1}/{p_j} \le {1}/{2} \le 
{1}/{p_k}$. 
In this case, 
the best choice of $(q_\ell)$ is 
${1}/{q_i}={1}/{q_j}={1}/{2}$ and 
${1}/{q_k} = 
1- {1}/{p_k}$ 
and the corresponding $q$ is 
${1}/{q}= {1}/{(2p_k)}$. 

{Case (II$\ast$)}   
${1}/{p_i}, {1}/{p_j} \le {1}/{2} \le 
{1}/{p_k} < 1$. 
In this case, 
the $(q_\ell)$ of Case (II) satisfies 
(A-2$\ast$) and (A-3) as well and hence 
${1}/{q}= {1}/{(2p_k)}$ satisfies 
\eqref{eq-inclusion-bbb}. 

{Case (III-1)}  
${1}/{p_i}\le {1}/{2} \le {1}/{p_j}, {1}/{p_k}$ 
and 
$(1- {1}/{p_j})+ (1- {1}/{p_k}) \ge {1}/{2}$. 
In this case, 
the best choice of $(q_{\ell})$ is 
${1}/{q_i}={1}/{2}$, ${1}/{q_j}=1- {1}/{p_j}$, 
and ${1}/{q_k}=1- {1}/{p_k}$. 
The corresponding $q$ is 
${1}/{q}=({1}/{2}) \big( {1}/{p_j }+ {1}/{p_k}-{1}/{2}\big)$. 

{Case (III-1$\ast$)} 
${1}/{p_i}\le {1}/{2} \le {1}/{p_j}, {1}/{p_k}$ and 
$(1- {1}/{p_j})+ (1- {1}/{p_k}) >  {1}/{2}$. 
In this case, 
the $(q_{\ell})$ chosen in Case (III-1) 
satisfies 
(A-2$\ast$) and (A-3) as well, 
and hence the 
corresponding $q$ satisfies 
\eqref{eq-inclusion-bbb}. 

{Case (III-2)}   
${1}/{p_i}\le {1}/{2} \le {1}/{p_j}, {1}/{p_k}$ 
and 
$(1- {1}/{p_j})+ (1- {1}/{p_k}) \le  {1}/{2}$. 
In this case, 
the best choice of $(q_{\ell})$ is 
${1}/{q_i}=
(1- {1}/{p_j})+ 
(1- {1}/{p_k})$, 
${1}/{q_j}=1- {1}/{p_j}$, 
and  
${1}/{q_k}=
1- {1}/{p_k}$. 
The corresponding $q$ is 
${1}/{q}= {1}/{p_j}+ {1}/{p_k} - 1$. 

{Case (IV-1)}  
${1}/{2} \le {1}/{p_1}, {1}/{p_2}, {1}/{p_3}$ and 
\begin{align*}
&
(1- {1}/{p_1})+ 
(1- {1}/{p_2}) 
\ge 1- {1}/{p_3}, 
\\
&
(1- {1}/{p_2})+ 
(1- {1}/{p_3}) 
\ge 1- {1}/{p_1}, 
\\
&
(1- {1}/{p_3})+ 
(1- {1}/{p_1})  
\ge 1- {1}/{p_2}.  
\end{align*}
In this case, 
the best choice of $(q_{\ell})$ is 
$
{1}/{q_1}
=1- {1}/{p_1}$, 
${1}/{q_2}=
1- {1}/{p_2}$, and 
${1}/{q_3}=
1- {1}/{p_3}$. 
The corresponding $q$ is 
$
{1}/{q}=({1}/{2}) \big( {1}/{p_1 }+ {1}/{p_2} 
+ {1}/{p_3}- 1
\big)$. 

{Case (IV-1$*$)}  
${1}/{2} \le {1}/{p_1}, {1}/{p_2}, {1}/{p_3}$ and 
\begin{align*}
&
(1- {1}/{p_1})+ 
(1- {1}/{p_2}) 
> 1- {1}/{p_3}, 
\\
&
(1- {1}/{p_2})+ 
(1- {1}/{p_3}) 
> 1- {1}/{p_1}, 
\\
&
(1- {1}/{p_3})+ 
(1- {1}/{p_1}) 
> 1- {1}/{p_2}, 
\\
&
 (1- {1}/{p_1})+ 
(1- {1}/{p_2}) 
+ (1- {1}/{p_3}) 
\ge 1. 
\end{align*}
In this case, 
the $(q_{\ell})$ chosen in Case (IV-1) 
satisfies 
(A-2$\ast$) and (A-3) as well, 
and hence the 
corresponding $q$ satisfies 
\eqref{eq-inclusion-bbb}. 

{Case (IV-2)}  
${1}/{2} \le {1}/{p_1}, {1}/{p_2}, {1}/{p_3}$ 
and 
$(1- {1}/{p_i})+ 
(1- {1}/{p_j}) 
\le 
1- {1}/{p_k}$. 
In this case, the best choice of $(q_{\ell})$ 
is 
${1}/{q_i}=
1- {1}/{p_i}$, 
${1}/{q_j}=
1- {1}/{p_j}$, and 
${1}/{q_k}=
(1- {1}/{p_i})+(1- {1}/{p_j})$. 
The corresponding 
$q$ is 
${1}/{q}= {1}/{p_i}+ {1}/{p_j} - 1$. 
\end{proof}

\subsection{Further results on the class $\calBform_{q_1, q_2, q_3}$}
\label{Further}
Here we shall give several results 
concerning the class $\calBform_{q_1, q_2, q_3}$ 
that are not necessarily covered by 
Proposition \ref{thellq-ellqweak-calB}.

\begin{prop}\label{th-22pprime}
Let $1 < p <\infty$.
Then 
$
\ell^{2p,\infty}_{+} (\Z^n \times \Z^n)
\subset 
\calBform_{p^{\prime},2,2} 
\cap \calBform_{2,p^{\prime},2} 
\cap \calBform_{2,2,p^{\prime}}.
$
\end{prop}

\begin{proof}
The inclusion $\ell^{2p,\infty}_{+} (\Z^n \times \Z^n)
\subset 
\calBform_{p^{\prime},2,2} $ 
follows 
from the inequality \eqref{eqRnBLLorentz} with 
$q_0=2p$, 
$q_1=p^{\prime}$, 
$q_2=q_3=2$, 
$r_0=\infty$, 
$r_1=p^{\prime}$, 
and 
$r_2=r_3=2$.
The assertion for the class $\calBform_{2,p',2}$
follows by symmetry. 
The inclusion $\ell^{2p,\infty}_{+} (\Z^n \times \Z^n)
\subset 
\calBform_{2, 2, p^{\prime}} $ 
follows 
from the inequality \eqref{eqRnBLLorentz} with 
$q_0=2p$, 
$q_1=q_2=2$, 
$q_3=p^{\prime}$, 
$r_0=\infty$, 
$r_1=r_2=2$, 
and 
$r_3=p^{\prime}$.
\end{proof}

\begin{prop}\label{th-22infty}
The following 
functions on $\Z^n \times \Z^n$ 
belong to 
$
\calBform_{2,2,\infty} \cap 
\calBform_{2,\infty,2} \cap 
\calBform_{\infty,2,2}.
$
\begin{enumerate}\setlength{\itemsep}{3pt} 
\item
$V_1(\nu_1,\nu_2) = \prod_{k=1}^{n}(1+|\nu_{1,k}|+|\nu_{2,k}|)^{-1}$, 
$\nu_j=(\nu_{j,1}, \cdots, \nu_{j,n}) \in \Z^n$, $j=1,2$.
\item
$V(\nu_1,\nu_2) = \langle (\nu_1,\nu_2) \rangle^{-n}$, 
$\nu_1, \nu_2 \in \Z^n$.
\item
Any nonnegative function in 
${\ell^{2}(\Z^n \times \Z^n)}$. 
\end{enumerate}
\end{prop}

\begin{proof}(1)
First, we shall prove the claim 
$V_1 \in \calBform_{2,2,\infty}$, 
which is equivalent to the inequality 
\begin{equation}\label{eqB22infty}
\sum_{\nu_1,\nu_2 \in \Z^n} 
V_1 (\nu_1, \nu_2) 
A(\nu_1) 
B(\nu_2)
\le c
\|A\|_{\ell^{2}(\Z^n)} 
\|B\|_{\ell^{2}(\Z^n)}. 
\end{equation}
In the one dimensional case, 
this inequality is   
known as the Hilbert inequality
(see, {\it e.g.}, \cite[Chapter IX, Theorem 315]{HLP}
or \cite[Appendix A.3]{grafakos 2014m}). 
In the $n$ dimensional case, we obtain 
\eqref{eqB22infty} by using the Hilbert inequality $n$ times. 
Next, 
the claim 
$V_1 \in \calBform_{2,\infty ,2}$  
is equivalent to the inequality 
\begin{equation}\label{eqB2infty2}
\sum_{\nu_1,\nu_2\in \Z^n} 
V_1 (\nu_1, \nu_2) 
A(\nu_1) 
C(\nu_1+\nu_2)
\le 
c 
\|A\|_{\ell^{2}}
\|C\|_{\ell^{2}}. 
\end{equation}
However, by a simple change of variables 
and by the fact 
$
V_1 (\nu_1, \nu_2 - \nu_1) 
\approx 
V_1 (\nu_1, \nu_2) 
$, 
the inequality 
\eqref{eqB2infty2} is 
equivalent to \eqref{eqB22infty}, 
which is already proved. 
Finally, by symmetry 
we also have 
$V_1 \in \calBform_{\infty,2, 2}$. 


(2)
Since 
$V(\nu_1,\nu_2) \lesssim V_{1} (\nu_1,\nu_2)$,
the desired result follows from (1).

(3)
It is sufficient to prove 
the inequalities \eqref{eqB22infty} and \eqref{eqB2infty2} 
with $V_1$ replaced by 
$W\in \ell^2_{+} (\Z^n \times \Z^n)$. 
But both inequalities 
with $c=\|W\|_{\ell^2}$ 
immediately follow by an application of the Cauchy--Schwarz inequality. 
\end{proof}

\begin{prop}\label{th-p1p2q-prime}
Let $1 \le p_1,p_2 \le 2$ and 
${1}/{q}={1}/{p_1}+{1}/{p_2}-1$. 
Then 
$\ell^{q}_{+}(\Z^n \times \Z^n)
\subset  \calBform_{p_1^{\prime},p_2^{\prime},q^{\prime}}$. 
If in addition $1/p_1+1/p_2 \ge 3/2$,
then 
$\ell^{q}_{+}(\Z^n \times \Z^n)$ 
belong to $\calBform_{p_1',p_2',2} (\Z^n \times \Z^n)$.
\end{prop}

\begin{proof}
The former assertion follows from 
the inequality 
\eqref{eqRnBL} with 
$q_0=q$,
$q_1=p_1'$, 
$q_2=p_2'$,
and $q_3=q^{\prime}$. 
The latter assertion follows from the former one 
since $\calBform_{p_1^{\prime},p_2^{\prime},q^{\prime}}
\subset \calBform_{p_1',p_2',2} $ 
for $q^{\prime} \ge 2$. 
\end{proof}

Recall that, in the region 
$1\le p_1, p_2 \le 2$, 
the weight function corresponding to the 
critical bilinear H\"ormander class 
$BS^{m(p_1, p_2, p)}_{0,0}(\R^n)$ 
is 
$(1+|\nu_1|+|\nu_2|)^{-n(1/p_1 + 1/p_2 -1/2)}$ 
(see \eqref{eqcriticalorder}). 
The next proposition 
shows that this weight 
does not belong 
to the class $\calBform_{p_1',p_2',2}$ 
if $1/p_1+1/p_2 > 3/2$.

\begin{prop}\label{thcounterexBp1p22}
Let 
$1 \leq p_1,p_2 < 2$ and
$1/p_1+1/p_2 > 3/2$ 
and let 
${1}/{q}=({1}/{2}) ({1}/{p_1}+{1}/{p_2}-{1}/{2} )$.
Then 
$(1+|\nu_1|+|\nu_2|)^{-2n/q} 
\not\in 
\calBform_{p_1',p_2',2}(\Z^n \times \Z^n)$.
\end{prop}

\begin{proof}
The idea of this proof comes from \cite[Section 4]{Furuya-2017}.
We write 
$V(\nu_1,\nu_2)=(1+|\nu_1|+|\nu_2|)^{-2n/q}$ 
and prove that 
$V \not\in \calBform_{p_1',p_2',2}(\Z^n \times \Z^n)$.

We first consider the case $1<p_1,p_2 < 2$.
What we have to prove is that the inequality 
\begin{equation}\label{eqcontiBL}
\sum_{\nu_1, \nu_2 \in \Z^n}
V(\nu_1, \nu_2) 
A(\nu_1) 
B(\nu_2)
C(\nu_1+\nu_{2})
\lesssim
\|A\|_{\ell^{p_1'} (\Z^n) } 
\|B\|_{\ell^{p_2'} (\Z^n) } 
\|C\|_{\ell^{2} (\Z^n) } .
\end{equation}
does not hold. 
For this, consider the 
functions 
\begin{align*}
A(x)&=|x|^{-n/p_1'} \, \big(\log|x| \big)^{-1/\alpha p_1'} \ichi_{|x|\geq10}(x), \\
B(x)&=|x|^{-n/p_2'} \, \big(\log|x| \big)^{-1/\alpha p_2'} \ichi_{|x|\geq10}(x), \\
C(x)&=|x|^{-n/2} \, \big(\log|x| \big)^{-1/2\alpha } \ichi_{|x|\geq10}(x),
\end{align*}
where $\alpha = 1/{p_1'}+1/{p_2'}+1/{2}$.
Notice that $0<\alpha<1$ and hence 
$A \in L^{p_1'}$,
$B \in L^{p_2'}$,
and $C \in L^{2}$. 
Hence the right hand side of 
\eqref{eqcontiBL} is finite. 
We shall see that 
the left hand side 
of \eqref{eqcontiBL}, 
which we shall write $I$, 
is infinity. 
We have  
\begin{align*}
I\approx 
\int_{\R^{2n}} \frac{A(\nu_1)B(\nu_2)C(\nu_1+\nu_2)}
{(|\nu_1|+|\nu_2|)^{ n({1}/{p_1}+{1}/{p_2}-{1}/{2} )}} 
\,d\nu_1d\nu_2
&=
\int_{\R^{2n}} \frac{A(\nu_1)B(\nu_2)C(\nu_1+\nu_2)}
{(|\nu_1|+|\nu_2|)^{ 2n-\alpha n}} 
\,d\nu_1d\nu_2. 
\end{align*}
Since 
$1 = {1}/{(\alpha p_1')}+{1}/{(\alpha p_2')}+{1}/{(2\alpha)}$
and since 
\[
\log|\nu_1| ,\; \log|\nu_2| ,\; \log|\nu_1+\nu_2| \le
\log\big(|\nu_1|+|\nu_2| \big)
\]
for $|\nu_1|,|\nu_2|, |\nu_1 + \nu_2| \ge 10$, 
we have
\begin{align*}
I \gtrsim   
\int_{|\nu_1|,|\nu_2|,|\nu_1+\nu_2| \ge 10}
\frac{1}{ (|\nu_1|+|\nu_2|)^{2n} \log\big(|\nu_1|+|\nu_2| \big) } 
\,d\nu_1d\nu_2.
\end{align*}
Let 
\begin{align*}
D &:= \left\{ (\nu_1,\nu_2)\in\R^{2n} : \; |\nu_1|,|\nu_2| 
\ge 10, \; \; \nu_{1,j},\nu_{2,j} \geq 0, \;\; 
j=1,\cdots n \right\}
\\& \quad\subset
\left\{ (\nu_1,\nu_2)\in\R^{2n} : \; 
|\nu_1|,|\nu_2|,|\nu_1+\nu_2| \ge 10 \right\}, 
\end{align*}
where 
$\nu_i=(\nu_{i,1},\cdots,\nu_{i,n})\in\R^n$. 
Then 
\begin{align*}
I \gtrsim  
\int_{(\nu_1,\nu_2)\in D}
\frac{1}{ (|\nu_1|+|\nu_2|)^{2n} \log\big(|\nu_1|+|\nu_2| \big) } 
\,d\nu_1d\nu_2.
\end{align*}
By using the spherical coordinates for $\nu_1,\nu_2$,
we obtain
\begin{align*}
I
&
\gtrsim 
\int_{10}^{\infty} \int_{10}^{\infty}
\frac{r^{n-1}\,R^{n-1}}{ (r+R)^{2n} \log\big(r+R\big) } 
\,drdR
\\
&
\ge  
\int_{10}^{\infty} \left( \int_{R}^{2R}
\frac{r^{n-1}\,R^{n-1}}{ (r+R)^{2n} \log\big(r+R\big) } 
\,dr \right) dR
\approx
\int_{10}^{\infty} \frac{1}{ R \log R } 
\,dR
=\infty.
\end{align*}
Thus the inequality \eqref{eqcontiBL} does not hold. 
This completes the proof for the case 
$1<p_1,p_2 < 2$.

We next consider the case 
$1<p_1<2$ and $p_2=1$. 
We shall prove that
$V(\nu_1, \nu_2) \not\in 
\calBform_{p_1',\infty,2}(\Z^n \times \Z^n)$, 
or equivalently 
that 
the inequality 
\begin{align*}
\sum_{\nu_1, \nu_2 \in \Z^n} 
V(\nu_1, \nu_2) 
A(\nu_1) 
C(\nu_1+\nu_{2})
\le c
\|A\|_{\ell^{p_1'}} 
\|C\|_{\ell^{2}} 
\end{align*}
does not hold. 
Since 
$(1+|\nu_1|+|\nu_2|) \approx (1+|\nu_1|+|\nu_1+\nu_2|)$, 
the above inequality 
is equivalent to 
\begin{align}\label{eqp_2=1}
\sum_{\nu_1, \nu_2 \in \Z^n} 
\frac{A(\nu_1) C(\nu_{2})}
{(1+|\nu_1|+|\nu_2|)^{n(1/p_1+1/2)}}
\lesssim
\|A\|_{\ell^{p_1'}} 
\|C\|_{\ell^{2}} .
\end{align}
We shall prove that \eqref{eqp_2=1} does not hold. 
For this, consider the functions 
\begin{align*}
A(x)&=|x|^{-n/p_1'} \, 
\big(\log|x| \big)^{-1/\alpha p_1'} \ichi_{|x|\geq10}(x), 
\\
C(x)&=|x|^{-n/2} \, 
\big(\log|x| \big)^{-1/2\alpha } \ichi_{|x|\geq10}(x).
\end{align*}
where $\alpha = 1/{p_1'}+1/{2}$.
Note that $0<\alpha<1$.
Since these are the same functions as used in the previous case,
repeating the same argument as above,
we see that 
the right hand side of \eqref{eqp_2=1} is finite
but the left hand side is infinity.
This proves the case $1<p_1<2$ and $p_2=1$.

The case $p_1=1$ and $1<p_2<2$ 
follows from symmetry.

Finally, consider the case $p_1=p_2=1$. 
We shall prove 
$V(\nu_1, \nu_2)=(1+|\nu_1|+|\nu_2|)^{-3n/2} \not\in 
\calBform_{\infty,\infty,2}(\Z^n \times \Z^n)$.
By the fact
$(1+|\nu_1|+|\nu_2|) \approx (1+|\nu_1|+|\nu_1+\nu_2|)$,
it is sufficient to show that 
the inequality 
\begin{equation}\label{eqp1=p2=1}
\sum_{\nu_1, \nu_2 \in \Z^n} 
\frac{C(\nu_{2})}{(1+|\nu_1|+|\nu_2|)^{3n/2}}
\lesssim
\|C\|_{\ell^{2}} 
\end{equation}
does not hold. 
Here, since
\[
\sum_{\nu_1 \in \Z^n} 
\frac{1}{(1+|\nu_1|+|\nu_2|)^{3n/2}}
\approx 
\frac{1}{(1+|\nu_2|)^{n/2}}, 
\]
the inequality \eqref{eqp1=p2=1} 
is further equivalent to
\begin{align*}
\sum_{\nu_2 \in \Z^n} 
\frac{C(\nu_{2})}{(1+|\nu_2|)^{n/2}}
\lesssim
\|C\|_{\ell^{2}} .
\end{align*}
However, the last inequality does not hold since 
$(1+|\nu_2|)^{-n/2} \not\in \ell^{2}(\Z^n)$. 
Now the proof of Proposition \ref{thcounterexBp1p22} is complete.
\end{proof}

\subsection{Weight of product form}

In this subsection, 
we consider weight functions 
of the form 
$ (f_1 \otimes f_2) (\nu_1, \nu_2)=f_1(\nu_1)f_2(\nu_2)$, 
where 
$f_1$ and $f_2$ are nonnegativew functions on $\Z^n$. 
We show inequalities of the form 
\[
\| f_1 \otimes f_2 \|_{\calBform_{q_1,q_2,q_3}}
\lesssim 
\| f_1 \|_{\ell^{r_1,\infty}}
\| f_2 \|_{\ell^{r_2,\infty}}
\quad\text{or}\quad
\| f_1 \otimes f_2 \|_{\calBform_{q_1,q_2,q_3} }
\lesssim 
\| f_1 \|_{\ell^{r_1}}
\| f_2 \|_{\ell^{r_2}},  
\]
which will give some weight functions that are not covered 
by Proposition \ref{thellq-ellqweak-calB} nor by  
the propositions of Subsection \ref{Further}. 
The argument below is a generalization of the one 
given in \cite[Proposition 3.3 or Remark 3.6]{KMT-arxiv},
where 
the case 
$q_1=q_2=q=2$ was treated.

Our argument is based on the following lemma. 

\begin{lem}\label{thBrascamp-Lieb-prod}
Let 
$\widetilde{q_1}, 
\widetilde{q_2}, 
q_1, q_2, q_3, 
s_1, s_2, t_1, t_2, t_3 \in [1, \infty]$. 
Then the following hold. 
\\
$(1)$ The inequality 
\begin{equation}\label{eqRnBL-prod}
\begin{split}
&
\int_{\R^n \times \R^n}
f_1 (x_1) f_2 (x_2)
A(x_1)
B(x_2)
C(x_1+x_2)\, dx_1 dx_2 
\\
&\lesssim
\|f_1\|_{L^{\widetilde{q_1}} (\R^{n} ) }
\|f_1\|_{L^{\widetilde{q_2}} (\R^{n} ) }
\|A\|_{L^{q_1} (\R^{n}) }
\|B\|_{L^{q_2} (\R^{n}) }
\|C\|_{L^{q_3} (\R^{n}) }
\end{split}
\end{equation}
holds for all nonnegative measurable functions 
on $\R^n$ 
if  
\begin{align}
& 
{1}/{\widetilde{q_1}} + {1}/{\widetilde{q_2}} 
+ {1}/{q_1}
+ {1}/{q_2}
+ {1}/{q_3} =2, 
\label{eqLqeq-prod}
\\
& 
0\le {1}/{\widetilde{q_i}}+ {1}/{q_i} \le 1, 
\quad i = 1, 2. 
\nonumber
\end{align}  
\\
$(2)$ The inequality 
\begin{equation}\label{eqRnBLLorentz-prod}
\begin{split}
&
\int_{\R^n \times \R^n}
f_1(x_1) 
f_2 (x_2) 
A(x_1)
B(x_2)
C(x_1+x_2)\, dx_1 dx_2 
\\
&
\lesssim
\|f_1\|_{L^{\widetilde{q_1}, s_1} (\R^{n} ) }
\|f_2\|_{L^{\widetilde{q_2}, s_2} (\R^{n} ) }
\|A\|_{L^{q_1, t_1} (\R^{n}) }
\|B\|_{L^{q_2, t_2} (\R^{n}) }
\|C\|_{L^{q_3, t_3} (\R^{n}) }
\end{split}
\end{equation}
holds for all nonnegative measurable functions 
$\R^n$ 
if $(\widetilde{q_i})$ and $(q_i)$ satisfy \eqref{eqLqeq-prod} and
\begin{align*}
&
0< {1}/{\widetilde{q_1}},  {1}/{\widetilde{q_2}}, 
{1}/{q_1}, {1}/{q_2}, {1}/{q_3} <1, 
\\
&0< {1}/{\widetilde{q_i}}+ {1}/{q_i} < 1, 
\quad i = 1, 2,  
\end{align*}
and if $(s_i)$ and $(t_i)$ satisfy
\begin{equation*}\label{eqHolder2}
{1}/{s_1} + {1}/{s_2}
+
{1}/{t_1}+ 
{1}/{t_2}+
{1}/{t_3}
\ge 1. 
\end{equation*}
$(3)$ The inequality 
\begin{equation}\label{eqRnBLLorentz-prod-2}
\begin{split}
&
\int_{\R^n \times \R^n}
f_1(x_1) 
f_2 (x_2) 
A(x_1)
C(x_1+x_2)\, dx_1 dx_2 
\\
&
\lesssim
\|f_1\|_{L^{\widetilde{q_1}, s_1} (\R^{n} ) }
\|f_2\|_{L^{\widetilde{q_2}, s_2} (\R^{n} ) }
\|A\|_{L^{q_1, t_1} (\R^{n}) }
\|C\|_{L^{q_3, t_3} (\R^{n}) }
\end{split}
\end{equation}
holds for all nonnegative measurable functions 
$\R^n$ 
if $(\widetilde{q_i})$ and $(q_i)$ satisfy 
\begin{align*}
&
1/\widetilde{q_1} 
+ 1/\widetilde{q_2}
+1/q_1 + 1/q_3 =2, 
\\
&
0< {1}/{\widetilde{q_1}},  {1}/{\widetilde{q_2}}, 
{1}/{q_1}, {1}/{q_3} <1, 
\\
&0< {1}/{\widetilde{q_1}}+ {1}/{q_1} < 1, 
\end{align*}
and if $(s_i)$ and $(t_i)$ satisfy 
\begin{equation*}
{1}/{s_1} + {1}/{s_2}
+
{1}/{t_1}+ 
{1}/{t_3}
\ge 1. 
\end{equation*}
\end{lem}

\begin{proof}
The claims (1) and (2) are given in 
\cite[Remark 3.6]{KMT-arxiv}. 
The claim (3) can be deduced from 
the case $q_2=\infty$ of (1) by the use of 
real interpolation for multilinear operators. 
See also the remarks given at the last part of 
Proof of Lemma \ref{thBrascamp-Lieb}. 
\end{proof}

\begin{prop}\label{th-22pprimeprod}
Let $1 < p < \infty$.
Then the following hold.
\begin{enumerate}\setlength{\itemsep}{3pt} 
\item
Let
$p<r_1<\infty$,
$2<r_2<\infty$, 
$1/r_1 + 1/r_2 = 1/p$, 
$f_1 \in \ell^{r_1, \infty}_{+}(\Z^n)$, and
$f_2 \in \ell^{r_2, \infty}_{+}(\Z^n)$. 
Then 
$f_1 (\nu_1) f_2 (\nu_2)$
belongs to
$\calBform_{p',2,2}(\Z^n\times \Z^n)$.
\item
Let
$2<r_1<\infty$, 
$p<r_2<\infty$, 
$1/r_1 + 1/r_2 = 1/p$, 
$f_1 \in \ell^{r_1, \infty}_{+}(\Z^n)$, and
$f_2 \in \ell^{r_2, \infty}_{+}(\Z^n)$. 
Then $f_1 (\nu_1) f_2 (\nu_2)$
belongs to
$\calBform_{2,p',2}(\Z^n\times \Z^n)$.
\item
Let
$2<r_1, r_2 < \infty$, 
$1/r_1 + 1/r_2 = 1/p$, 
$f_1 \in \ell^{r_1, \infty}_{+}(\Z^n)$, and
$f_2 \in \ell^{r_2, \infty}_{+}(\Z^n)$. 
Then 
$f_1 (\nu_1) f_2 (\nu_2)$
belongs to
$\calBform_{2,2,p^{\prime}}(\Z^n\times \Z^n)$.
\end{enumerate}
\end{prop}

\begin{proof}
The claim (1) follows from the inequality 
\eqref{eqRnBLLorentz-prod}
with 
$\widetilde{q_1}=r_1$, 
$\widetilde{q_2}=r_2$, 
$q_1=p^{\prime}$, 
$q_2=q_3=2$, 
$s_1=s_2=\infty$, 
$t_1 = p^{\prime}$, 
and
$t_2=t_3=2$. 
The claim (2) follows from (1) by symmetry. 
The claim (3) follows from the inequality 
\eqref{eqRnBLLorentz-prod} with 
$\widetilde{q_1}=r_1$, 
$\widetilde{q_2}=r_2$, 
$q_1=q_2=2$, 
$q_3=p'$,
$s_1=s_2=\infty$,
$t_1=t_2=2$, and
$t_3=p^{\prime}$.  
\end{proof}


\begin{prop}\label{th-22inftyprod}
\begin{enumerate}\setlength{\itemsep}{3pt} 
\item
Let
$1< r_2 < 2 < r_1 < \infty$, 
$1/r_1 + 1/r_2 = 1$, 
$f_1 \in \ell^{r_1, \infty}_{+}(\Z^n)$, and
$f_2 \in \ell^{r_2, \infty}_{+}(\Z^n)$.
Then 
$f_1 (\nu_1) f_2 (\nu_2)$
belongs to
$\calBform_{2,\infty,2}(\Z^n\times \Z^n)$.
\item
Let
$1< r_1 < 2 < r_2 < \infty$, 
$1/r_1 + 1/r_2 = 1$, 
$f_1 \in \ell^{r_1, \infty}_{+}(\Z^n)$, and
$f_2 \in \ell^{r_2, \infty}_{+}(\Z^n)$.
Then 
$f_1 (\nu_1) f_2 (\nu_2)$
belongs to
$\calBform_{\infty,2,2}(\Z^n\times \Z^n)$.
\item
Let
$1< r_2 < 2 < r_1 < \infty$, 
$1/r_1 + 1/r_2 = 1$, 
$f_1 \in \ell^{r_1, \infty}_{+}(\Z^n)$, 
$f_2 \in \ell^{r_2, \infty}_{+}(\Z^n)$. 
Then 
$f_1 (\nu_1) f_2 (\nu_1+\nu_2)$
belongs to
$\calBform_{2,2,\infty}(\Z^n\times \Z^n)$.
\item
Let
$1< r_1 < 2 < r_2 < \infty$, 
$1/r_1 + 1/r_2 = 1$, 
$f_1 \in \ell^{r_1, \infty}_{+}(\Z^n)$, and 
$f_2 \in \ell^{r_2, \infty}_{+}(\Z^n)$.
Then 
$f_1 (\nu_1+\nu_2) f_2 (\nu_2)$
belongs to
$\calBform_{2,2,\infty}(\Z^n\times \Z^n)$. 
\end{enumerate}
\end{prop}

\begin{proof}
The claim (1) follows from the inequality 
\eqref{eqRnBLLorentz-prod-2} with 
$\widetilde{q_1}=r_1$, 
$\widetilde{q_2}=r_2$, 
$q_1=q_3=2$, 
$s_1=s_2=\infty$, 
and 
$t_1=t_3=2$.  
The claim (2) follows from (1) by symmetry. 
The claims (3) and (4) follow from (1) and (2) 
by a simple change of variables. 
\end{proof}

\begin{prop}\label{th-p1p2q-primeprod}
\begin{enumerate}
\item
Let
$1 \leq p_1,p_2 \leq \infty$, 
$p_1 \leq r_1 \leq \infty$,
$p_2 \leq r_2 \leq \infty$, and
${1}/{r_1}+{1}/{r_2}
={1}/{p_1}+{1}/{p_2}-{1}/{2}$. 
Suppose that
$f_1 \in \ell^{r_1}_{+}(\Z^n)$ and
$f_2 \in \ell^{r_2}_{+}(\Z^n)$.
Then $f_1 (\nu_1) f_2 (\nu_2)$
belongs to
$\calBform_{p_1',p_2',2}(\Z^n \times \Z^n)$.
\item
Let
$1 < p_1,p_2 < \infty$, 
$1/p_1+1/p_2 \le 3/2$,
$p_1 < r_1 < \infty$,
$p_2 < r_2 < \infty$, and
${1}/{r_1}+{1}/{r_2}={1}/{p_1}+{1}/{p_2}-{1}/{2}$.
Suppose that 
$f_1 \in \ell^{r_1,\infty}_{+}(\Z^n)$ and
$f_2 \in \ell^{r_2,\infty}_{+}(\Z^n)$. 
Then $f_1 (\nu_1) f_2 (\nu_2)$
belongs to
$\calBform_{p_1',p_2',2}(\Z^n \times \Z^n)$.
\end{enumerate}
\end{prop}

\begin{proof}
The claim (1) follows from the inequality 
\eqref{eqRnBL-prod} with 
$\widetilde{q_1}=r_1$, 
$\widetilde{q_2}=r_2$, 
$q_1=p_1^{\prime}$, 
$q_2=p_2^{\prime}$, 
$q_3=2$.  
The claim (2) follows from the inequality 
\eqref{eqRnBLLorentz-prod} with 
$\widetilde{q_1}=r_1$, 
$\widetilde{q_2}=r_2$, 
$s_1 = s_2 = \infty$, 
$q_1=t_1=p_1^{\prime}$, 
$q_2=t_2=p_2^{\prime}$, 
$q_3=t_3=2$. 
\end{proof}

\section{Proofs of the main results} 
\label{sectionProof}
\subsection{Proof of the `if' part of Theorem \ref{thmain-thm-WA}}

We shall prove the 
proposition below, which essentially 
contains the `if' part of Theorem \ref{thmain-thm-WA}.

\begin{prop}
\label{thmain-prop}
Let $p_1,p_2,p,q_1,q_2,q \in (0,\infty]$
and
$1/p = 1/p_1+1/p_2$.
Assume that $V \in \calM (\R^{2n})$
and the restriction of $V$ to
$\Z^n \times \Z^n$ belongs to
$\calB_{q_1,q_2,q}(\Z^n \times \Z^n)$.
Then  all 
$T_{\sigma} \in \mathrm{Op}
(BS^{V}_{0,0} (\R^n))$
are bounded from
$W^{p_1,q_1} \times W^{p_2,q_2}$
to $W^{p,q}$.
\end{prop}

Before we prove this proposition, 
we shall see that the `if' part of Theorem \ref{thmain-thm-WA}  
follows from it.

\begin{proof}[Proof of the `if' part of Theorem \ref{thmain-thm-WA}]
Suppose $p_1,p_2,p,q_1,q_2,q \in (0,\infty]$, 
$V \in \calB_{q_1,q_2,q}(\Z^n \times \Z^n)$,  
and
$\sigma \in BS^{ \widetilde{V} }_{0,0} (\R^n)$.
Take a function $V^{\ast} \in \calM(\R^{2n})$ 
satisfying the properties as mentioned in Proposition \ref{thVast}.
Then $\widetilde{V} \lesssim V^{\ast}$
and thus $\sigma \in BS^{ V^{\ast} }_{0,0} (\R^n)$.
Since $V^{\ast}|_{\Z^n \times \Z^n}$ belongs to 
$\calB_{q_1, q_2, q}$, Proposition \ref{thmain-prop} implies 
the boundedness 
$T_\sigma : W^{p_1,q_1} \times W^{p_2,q_2} 
\to W^{p,q}$ 
for $1/p=1/p_1+1/p_2$.
Since $W^{r,q} \hookrightarrow W^{s,q}$ for $r \le s$
by Proposition \ref{thWaembd} (1),
the boundedness also holds
for 
$1/p \le 1/p_1+1/p_2$.
\end{proof}

Now, we shall prove Proposition \ref{thmain-prop}.
For this, we use the following lemma.  

\begin{lem}[{P\"aiv\"arinta--Somersalo \cite[Lemma 2.2]{PS-1988}}]
\label{thPSdec}
For each $N\in \N$, 
there exists a sequence of functions 
$\{ \chi_{\ell} \}_{\ell \in \Z^n} \subset \calS (\R^n)$ 
such that 
\begin{align*}
&\supp \chi_{\ell} \in \ell + [-1, 1]^n, 
\qquad 
\sup_{\ell} \| \calF^{-1} \chi_\ell \|_{L^1} \le c, \\
&\sum_{\ell \in \Z^n} \langle \ell \rangle^{-2N} \langle \zeta \rangle^{2N} \chi_\ell (\zeta) = 1
\quad\textrm{for any}\quad
\zeta \in \R^n, 
\end{align*}
where $c$ is a constant depending only on $N$ and $n$. 
\end{lem}

\begin{proof}[Proof of Proposition \ref{thmain-prop}]
Let $\sigma \in BS^{V}_{0,0}(\R^n)$. 
We first rewrite 
$T_\sigma (f_1,f_2)$
by using the Fourier series expansion 
(this method was used by 
Coifman and Meyer \cite{CM-Ast, CM-AIF}).

Take a function 
$\varphi$ 
such that 
\begin{equation*}
\varphi \in \calS (\R^n), 
\quad 
\supp \varphi \subset [-1, 1]^{n}, \quad
\sum_{\ell\in\Z^n} \varphi (\zeta-\ell) = 1
\quad\textrm{for all}\quad
\zeta \in \R^n.  
\end{equation*}
Using $\varphi$, 
we decompose the symbol as
\[
\sigma(x,\xi_1,\xi_2)
=
\sum_{ (\nu_1,\nu_2) \in \Z^n \times \Z^n } 
\sigma(x,\xi_1,\xi_2)\varphi (\xi_1-\nu_1)\varphi (\xi_2-\nu_2)
=
\sum_{\boldsymbol{\nu} \in \Z^n \times \Z^n} 
\sigma_{\boldsymbol{\nu}}(x,\xi_1,\xi_2), 
\]
where 
\[
\sigma_{\boldsymbol{\nu}}(x,\xi_1,\xi_2)=
\sigma(x,\xi_1,\xi_2)\varphi (\xi_1-\nu_1)\varphi (\xi_2-\nu_2), \qquad
\boldsymbol{\nu}=(\nu_1,\nu_2) \in \Z^n \times \Z^n.
\]
Thus 
$T_{\sigma}=\sum_{\boldsymbol{\nu}} 
T_{\sigma_{\boldsymbol{\nu}}}$. 
Since $V\in\calM (\R^{2n})$ and 
$\varphi \in \calS (\R^n)$, 
we have 
\begin{equation}\label{eqWnu}
\sup_{ x,\xi_1,\xi_2 \in \R^n }
\left|
\partial^{\alpha}_x 
\partial^{\beta_1}_{\xi_1} 
\partial^{\beta_2}_{\xi_2} 
\sigma_{\boldsymbol{\nu}} (x,\xi_1,\xi_2)
\right|
\lesssim
V (\boldsymbol{\nu})
\end{equation}
for all multi-indices $\alpha,\beta_1,\beta_2 \in (\N_0)^n$.

We take a function $\widetilde\varphi \in \calS (\R^n)$ 
such that 
\[
\widetilde\varphi = 1 \;\; \text{on}\;\; [-1, 1]^n, 
\quad 
\supp \widetilde\varphi \subset [-2 , 2]^{n}, 
\quad 
\bigg| \sum_{k\in \Z^n} \widetilde\varphi(\cdot-k) \bigg| \geq 1. 
\]
Then,  
extending 
the function 
$\sigma_{\boldsymbol{\nu}} (x,\cdot ,\cdot )|_{
\boldsymbol{\nu} + [-\pi , \pi ]^{2n}}$ 
to a $2\pi  \Z^{2n}$-periodic function 
on $\R^{n}\times \R^n$ 
and using Fourier series expansion, 
we can write 
$\sigma_{\boldsymbol{\nu}} (x,\xi_1,\xi_2)$ as 
\begin{align*}
\sigma_{\boldsymbol{\nu}} (x,\xi_1,\xi_2)
= \sum_{\boldsymbol{k}=(k_1,k_2) \in\Z^n \times \Z^n}\, 
P_{\boldsymbol{\nu},\boldsymbol{k}} (x)\, 
e^{i (k_1 \cdot \xi_1 + k_2 \cdot \xi_2)}\, 
\widetilde\varphi(\xi_1-\nu_1) \widetilde\varphi(\xi_2-\nu_2) 
\end{align*}
with 
\begin{align*}
P_{\boldsymbol{\nu},\boldsymbol{k}} (x) &= 
\frac{1}{(2 \pi )^{2n}} \int_{\boldsymbol{\nu}+[-\pi , \pi ]^{2n}}
e^{-i (k_1 \cdot \eta_1 + k_2 \cdot \eta_2)}
\sigma_{\boldsymbol{\nu}} (x,\eta_1, \eta_2)
\,d\eta_1d\eta_2.
\end{align*}
Integration by parts gives
\begin{align*}
P_{\boldsymbol{\nu},\boldsymbol{k}} (x) 
&=
\frac{\langle \boldsymbol{k} \rangle^{-2M}}
{(2 \pi )^{2n}} \int_{\boldsymbol{\nu}+[-\pi , \pi ]^{2n}}
e^{-i (k_1 \cdot \eta_1 + k_2 \cdot \eta_2)}
\big( I-\Delta_{\eta_1,\eta_2} \big)^{M} \big\{ \sigma_{\boldsymbol{\nu}} (x,\eta_1, \eta_2) \big\}
\,d\eta_1d\eta_2
\\
&=
\langle  \boldsymbol{k} \rangle^{-2M}
Q_{\boldsymbol{\nu},\boldsymbol{k}} (x), 
\end{align*}
where 
\begin{equation}\label{eqQnuk}
\begin{split}
&
Q_{\boldsymbol{\nu},\boldsymbol{k}} (x) 
\\
&=
\frac{1}{(2 \pi )^{2n}} 
\int_{\boldsymbol{\nu}+[-\pi , \pi ]^{2n}}
e^{-i (k_1 \cdot \eta_1 + k_2 \cdot \eta_2)}
\big( I-\Delta_{\eta_1,\eta_2} \big)^{M} \big\{ \sigma_{\boldsymbol{\nu}} (x,\eta_1, \eta_2) \big\}
\,d\eta_1d\eta_2.  
\end{split}
\end{equation}
We further decompose 
$Q_{\boldsymbol{\nu},\boldsymbol{k}}$
by using Lemma \ref{thPSdec} as 
\begin{align*}
Q_{\boldsymbol{\nu},\boldsymbol{k}} (x) &= 
\sum_{\ell \in \Z^n} \langle \ell \rangle^{-2N} 
Q_{\boldsymbol{\nu},\boldsymbol{k},\ell} (x), \\
Q_{\boldsymbol{\nu},\boldsymbol{k},\ell} (x)&= 
\calF^{-1} \left[ \langle \zeta \rangle^{2N} 
\chi_\ell (\zeta) 
\widehat{Q_{\boldsymbol{\nu},\boldsymbol{k}}}(\zeta) \right](x).
\end{align*}
Thus 
we can express the symbol $\sigma_{\boldsymbol{\nu}} $ as 
\begin{align*}\label{eqsigmanu}
\sigma_{\boldsymbol{\nu}} (x,\xi_1,\xi_2)
= 
\sum_{
\substack{
\boldsymbol{k} 
\in \Z^n \times \Z^n, 
\\ 
\ell \in \Z^n}}
\, 
\langle  \boldsymbol{k} \rangle^{-2M}
\langle \ell \rangle^{-2N} 
Q_{\boldsymbol{\nu},\boldsymbol{k}, \ell} (x)\, 
e^{i (k_1 \cdot \xi_1 + k_2 \cdot \xi_2)}\, 
\widetilde\varphi(\xi_1-\nu_1) \widetilde\varphi(\xi_2-\nu_2). 
\end{align*}
From this representation,
the bilinear operator $T_{\sigma_{\boldsymbol{\nu}}}$
is written as 
\begin{equation*}\label{eqTsigmanu}
T_{\sigma_{\boldsymbol{\nu}}}(f_1,f_2)(x)
= 
\sum_{\boldsymbol{k}, \ell }
\langle \boldsymbol{k} \rangle^{-2M}
\langle \ell \rangle^{-2N} 
Q_{\boldsymbol{\nu},\boldsymbol{k}, \ell} (x)
\prod_{j=1,2}
\widetilde\varphi(D-\nu_j)f_j(x+ k_j). 
\end{equation*}
We write 
\[
F^{j}_{\nu_j,k_j}(x) = \widetilde\varphi(D-\nu_j)f_j(x+ k_j),
\quad
j=1,2. 
\]
Thus 
\begin{equation}\label{eqdecompositionTsnu}
T_{\sigma_{\boldsymbol{\nu}}}(f_1,f_2)(x)
= 
\sum_{\boldsymbol{k},\ell}
\langle \boldsymbol{k} \rangle^{-2M} \langle \ell \rangle^{-2N} 
Q_{\boldsymbol{\nu},\boldsymbol{k},\ell} (x)
\prod_{j=1,2}
F^{j}_{\nu_j,k_j}(x) .
\end{equation}

Now, we actually start the estimate for the boundedness.
We use the expression 
\eqref{eqdecompositionTsnu} with the integers $M$ and $N$ 
satisfying 
$2M > 2n/\min(1,p,q)$
and
$2N > n/\min(1,p,q)$.
Then, since 
$T_{\sigma} = \sum_{\boldsymbol{\nu}} 
T_{\sigma_{\boldsymbol{\nu}}}$,
from \eqref{eqdecompositionTsnu} we have
\begin{align*}
\begin{split}
&
\| T_{\sigma}(f_1,f_2) \|_{W^{p,q}} ^{\min(1,p,q)}
\\&\leq
\sum_{\boldsymbol{k},\ell}
\langle 
 \boldsymbol{k} \rangle^{-2M\min(1,p,q)} \langle \ell \rangle^{-2N\min(1,p,q)} 
\left\| \sum_{\boldsymbol{\nu}} 
Q_{\boldsymbol{\nu},\boldsymbol{k},\ell}
\prod_{j=1,2} 
F^{j}_{\nu_j,k_j}
\right\|_{W^{p,q}}^{\min(1,p,q)}
\\&\lesssim 
\sup_{\boldsymbol{k},\ell}
\left\| \sum_{\boldsymbol{\nu}} 
Q_{\boldsymbol{\nu},\boldsymbol{k},\ell}
\prod_{j=1,2} 
F^{j}_{\nu_j,k_j}
\right\|_{W^{p,q}}^{\min(1,p,q)}.
\end{split}
\end{align*}
In order to complete the proof 
of 
Proposition \ref{thmain-prop}, 
it is sufficient to prove the estimate 
\begin{align}\label{eqketsuron}
\sup_{\boldsymbol{k},\ell}
\left\| 
\sum_{\boldsymbol{\nu}} 
Q_{\boldsymbol{\nu},\boldsymbol{k},\ell}
\prod_{j=1,2} 
F^{j}_{\nu_j,k_j}
\right\|_{W^{p,q}}
\lesssim
\|V\|_{\calB_{q_1,q_2,q}}
\| f_1 \|_{W^{p_1,q_1}}
\| f_2 \|_{W^{p_2,q_2}}.
\end{align}

Notice that
\begin{equation}\label{eqsuppQF}
\supp \calF 
\left[
Q_{\boldsymbol{\nu},\boldsymbol{k},\ell}  
\prod_{j=1,2} F^{j}_{\nu_j,k_j}\right]
\subset \{\zeta\in\R^n: |\zeta - \ell - \nu_1 -\nu_2 | \lesssim 1 \} 
\end{equation}
since
$\supp \calF Q_{\boldsymbol{\nu},\boldsymbol{k},\ell}
\subset \{\zeta\in\R^n: |\zeta - \ell | \lesssim 1 \}$
and
$\supp \calF F^{j}_{\nu_j,k_j}
\subset \{\zeta\in\R^n: |\zeta - \nu_j | \lesssim 1 \}$.
Moreover,
the support of the function $\kappa \in \calS(\R^n)$ 
in the definition of $W^{p,q}$ is compact.
By these support conditions, we have
\[
\kappa(D-\mu)\left[
\sum_{\boldsymbol{\nu}} 
Q_{\boldsymbol{\nu},\boldsymbol{k},\ell}
\prod_{j=1,2} F^{j}_{\nu_j,k_j}\right]
=
\sum_{\boldsymbol{\nu}:|\nu_1+\nu_2+\ell-\mu|\lesssim 1} 
\kappa(D-\mu)\left[
Q_{\boldsymbol{\nu},\boldsymbol{k},\ell}
\prod_{j=1,2} F^{j}_{\nu_j,k_j}\right]. 
\] 
Thus
\begin{align}
&\left\| 
\sum_{\boldsymbol{\nu}} 
Q_{\boldsymbol{\nu},\boldsymbol{k},\ell}
\prod_{j=1,2} 
F^{j}_{\nu_j,k_j}
\right\|_{W^{p,q}}
\nonumber 
\\
&=
\left\| \left\| 
\sum_{\tau:|\tau|\lesssim 1} \, 
\sum_{\boldsymbol{\nu}: \nu_1+\nu_2+\ell-\mu=\tau} 
\kappa(D-\mu)\left[
Q_{\boldsymbol{\nu},\boldsymbol{k},\ell}
\prod_{j=1,2} F^{j}_{\nu_j,k_j}\right]
\right\|_{\ell^{q}_{\mu}(\Z^n)} \right\|_{L^{p}(\R^n)}
\nonumber
\\&\lesssim
\sum_{\tau:|\tau|\lesssim 1} 
\left\| \left\| 
\kappa(D-\mu-\ell+\tau)\left[
\sum_{\boldsymbol{\nu}:\nu_1+\nu_2=\mu}
Q_{\boldsymbol{\nu},\boldsymbol{k},\ell}
\prod_{j=1,2} F^{j}_{\nu_j,k_j}\right]
\right\|_{\ell^{q}_{\mu}(\Z^n)} \right\|_{L^{p}(\R^n)} .
\label{eqestimatethis}
\end{align}
We shall estimate 
\eqref{eqestimatethis}.

We write 
\[
h_{\mu}
=
h_{\mu, \boldsymbol{k},\ell}
=\sum_{\boldsymbol{\nu}:\nu_1+\nu_2=\mu}
Q_{\boldsymbol{\nu},\boldsymbol{k},\ell}
\prod_{j=1,2} F^{j}_{\nu_j,k_j} .
\]
Since $\supp \kappa$ is compact and since 
\eqref{eqsuppQF} holds, we have 
\[
\supp \calF \left[ \check\kappa(\cdot) \, h_\mu(x-\cdot) \right]
\subset \{\zeta \in \R^n : |\zeta+\mu+\ell| \lesssim 1\}.
\]
Hence, 
the Nikol'skij inequality
(see, {\it e.g.}, \cite[Proposition 1.3.2]{triebel 1983})
gives
\begin{equation}\label{eqNikolskij}
\left| \kappa(D-\mu-\ell+\tau)h_{\mu}(x) \right|
\le
\left\| \check\kappa(y) \, h_\mu(x-y) \right\|_{L^1_y}
\lesssim
\left\| \check\kappa(y) \, 
h_\mu(x-y) \right\|_{L^{\epsilon}_y}
\end{equation}
for any $\epsilon \in (0,1]$. 
Here it should be remarked that the implicit 
constant in \eqref{eqNikolskij} 
depends on the diameter of 
$\supp \calF 
\left[ \check\kappa(\cdot) \, h_\mu(x-\cdot) \right]$ 
but not on its location,
and hence 
the implicit constant is independent of 
$\mu, \boldsymbol{k}, 
\ell$, and $\tau$.
Using \eqref{eqNikolskij} with $\epsilon = \min (1, p, q)$ and 
using the Minkowski inequality for integral, 
we have 
\begin{align*}
\left\| \left\| 
\kappa(D-\mu-\ell+\tau)h_{\mu}
\right\|_{\ell^{q}_{\mu}} \right\|_{L^{p}}
&\lesssim
\left\| \left\| 
\left\| \check\kappa(y) \, h_\mu(x-y) 
\right\|_{L^{\epsilon }_y}
\right\|_{\ell^{q}_{\mu}} \right\|_{L^{p}_x}
\\&\leq
\left\|
\left\| \left\| 
\check\kappa(y) \, 
h_\mu(x-y) 
\right\|_{\ell^{q}_{\mu}} 
\right\|_{L^{p}_x}
\right\|_{L^{\epsilon }_y}
\approx
\left\| \left\| 
h_\mu 
\right\|_{\ell^{q}_{\mu}} \right\|_{L^{p}}.
\end{align*}
Thus 
we obtain
\begin{align}\label{eqmatome}
\left\| 
\sum_{\boldsymbol{\nu}} 
Q_{\boldsymbol{\nu},\boldsymbol{k},\ell}
\prod_{j=1,2} 
F^{j}_{\nu_j,k_j}
\right\|_{W^{p,q}}
\lesssim 
\left\| \left\| 
\sum_{\boldsymbol{\nu}:\nu_1+\nu_2=\mu}
Q_{\boldsymbol{\nu},\boldsymbol{k},\ell}
\prod_{j=1,2} F^{j}_{\nu_j,k_j}
\right\|_{\ell^{q}_{\mu}(\Z^n)} \right\|_{L^{p}(\R^n)}
\end{align}
with the implicit constant 
independent of $\boldsymbol{k}$ and $\ell$.

As for $Q_{\boldsymbol{\nu},\boldsymbol{k},\ell} $, 
we have 
\begin{align*}
\big\| Q_{\boldsymbol{\nu},\boldsymbol{k},\ell} \big\|_{L^\infty}
=
\left\| 
\calF^{-1} \left[ \chi_\ell (\zeta) 
\big( 
(I-\Delta)^{N} Q_{\boldsymbol{\nu},\boldsymbol{k}} 
\big)^{\wedge} (\zeta)\right]
\right\|_{L^\infty}
\leq
\| \calF^{-1} \chi_\ell \|_{L^1} 
\| (I-\Delta)^{N} Q_{\boldsymbol{\nu},\boldsymbol{k}}
 \|_{L^\infty}.
\end{align*}
Since 
$\sup_{\ell} \| \calF^{-1} \chi_\ell \|_{L^1} \lesssim 1$
by Lemma \ref{thPSdec} 
and since 
$
\| (I-\Delta_x)^{N} 
Q_{\boldsymbol{\nu},\boldsymbol{k}} (x)
\|_{L^\infty_x}
\lesssim V(\boldsymbol{\nu})
$
by \eqref{eqWnu} and \eqref{eqQnuk}, 
we obtain 
\begin{equation}\label{eqestimateQnukell}
\big\| Q_{\boldsymbol{\nu},\boldsymbol{k},\ell} 
\big\|_{L^\infty}
\lesssim V(\boldsymbol{\nu})
\end{equation}
with 
the implicit constant independent of $\boldsymbol{k}$ and $\ell$.

Now using \eqref{eqmatome},
\eqref{eqestimateQnukell}, 
the definition of 
$\|V \|_{\calB_{q_1,q_2,q}}$,
and the H\"older inequality for
$1/p_1+1/p_2=1/p$, 
we have 
\begin{align*}
&
\left\| 
\sum_{\boldsymbol{\nu}} 
Q_{\boldsymbol{\nu},\boldsymbol{k},\ell}
\prod_{j=1,2} 
F^{j}_{\nu_j,k_j}
\right\|_{W^{p,q}}\lesssim 
\left\| \left\| 
\sum_{\boldsymbol{\nu}:\nu_1+\nu_2=\mu}
V(\boldsymbol{\nu})
\prod_{j=1,2} \left| F^{j}_{\nu_j,k_j} \right|
\right\|_{\ell^{q}_{\mu}(\Z^n)} \right\|_{L^{p}(\R^n)} 
\\
&
\le
\left\| 
\|V\|_{\calB_{q_1,q_2,q}}
\prod_{j=1,2}
 \big\| F^{j}_{\nu_j,k_j}\big\|_{\ell^{q_j}_{\nu_j}(\Z^n)}
\right\|_{L^{p}(\R^n)}
\leq 
\|V\|_{\calB_{q_1,q_2,q}}
\prod_{j=1,2} 
\Big\| \big\| F^{j}_{\nu_j,k_j}\big\|_{\ell^{q_j}_{\nu_j}(\Z^n)} \Big\|_{L^{p_j}(\R^n)}.
\end{align*}
The choice of
$\widetilde\varphi \in \calS(\R^n)$ implies that
\begin{align*}
\Big\| \big\| F^{j}_{\nu_j,k_j}(x) \big\|_{\ell^{q_j}_{\nu_j}} \Big\|_{L^{p_j}_x}
= 
\Big\| \big\|
\widetilde\varphi(D-\nu_j)f_j(x+ k_j) 
\big\|_{\ell^{q_j}_{\nu_j}} \Big\|_{L^{p_j}_x}
\approx
\| f_j \|_{W^{p_j,q_j}}. 
\end{align*}
Thus we obtain 
\begin{align*}
\left\| 
\sum_{\boldsymbol{\nu}} 
Q_{\boldsymbol{\nu},\boldsymbol{k},\ell}
\prod_{j=1,2} 
F^{j}_{\nu_j,k_j}
\right\|_{W^{p,q}}
\lesssim
\|V\|_{\calB_{q_1,q_2,q}}
\| f_1 \|_{W^{p_1,q_1}}
\| f_2 \|_{W^{p_2,q_2}}
\end{align*}
with the implicit constant independent of $\boldsymbol{k}$ and $\ell$.
We have proved 
\eqref{eqketsuron} and 
the proof of 
Proposition \ref{thmain-prop} 
is complete. 
\end{proof}

\subsection{Proof of the `only if' part of Theorem \ref{thmain-thm-WA}}

The basic idea used here
goes back to \cite[Lemma 6.3]{MT-2013}. 

\begin{proof}[Proof of the `only if' part of Theorem \ref{thmain-thm-WA}] 
We assume all 
$T_{\sigma} \in \mathrm{Op}
(BS^{\widetilde{V}}_{0,0} (\R^n))$ 
are bounded from
$W^{p_1,q_1} \times W^{p_2,q_2}$
to $W^{p,q}$. 
Then by the closed graph theorem 
there exist a positive integer $M$
and a positive constant $C$ such that
\begin{align}\label{eqinequality001}
\begin{split}
\|T_{\sigma}\|_{ W^{p_1,q_1} \times W^{p_2,q_2} \to W^{p,q} } 
\le 
C\max_{|\alpha|, |\beta_1|, |\beta_2| \le M}
\left\| 
\widetilde{V}(\xi_1,\xi_2)^{-1} 
\partial^{\alpha}_x 
\partial^{\beta_1}_{\xi_1} \partial^{\beta_2}_{\xi_2} 
\sigma(x,\xi_1,\xi_2)
\right\|_{L^{\infty}}
\end{split}
\end{align}
for all bounded smooth functions $\sigma$ on $(\R^n)^{3}$ 
(see \cite[Lemma 2.6]{BBMNT} 
for the argument using the closed graph theorem).

To define the norm of the Wiener amalgam space, 
we use a real valued function 
$\kappa\in \calS (\R^n)$ such that 
\[
\kappa \geq 0, \quad
\kappa =1 \ \text{on $[-1/2,1/2]^n$}, \quad
\supp \kappa \subset [-3/4,3/4]^n, \quad
\sum_{k\in\Z^n} \kappa (\cdot-k) \geq 1.  
\]
We take functions 
$\varphi, \widetilde{\varphi} \in \calS (\R^n)$ such that 
\begin{align*}
& \mathrm{supp}\, \varphi \subset [-1/8, 1/8]^n, 
\quad 
\varphi \not\equiv 0, 
\\
&
\mathrm{supp}\, \widetilde{\varphi} 
\subset [-1/2,1/2]^n, 
\quad
\widetilde{\varphi}=1 
\;\;\text{on}\;\;
[-1/4,1/4]^n.  
\end{align*}
Define the symbol $\sigma$ by
\[
\sigma (x,\xi_1,\xi_2)
=\sigma (\xi_1,\xi_2)
=\sum_{ k_1,k_2 \in \Z^n }
V(k_1, k_2) 
\widetilde{\varphi}(\xi_{1}-k_{1})
\widetilde{\varphi}(\xi_{2}-k_{2}). 
\]
Obviously, we have
\begin{equation}\label{eqsigmatest}
\left| 
\partial^{\beta_1}_{\xi_1} \partial^{\beta_2}_{\xi_2} 
\sigma(\xi_1,\xi_2) 
\right|
\le C_{\beta_1,\beta_2}
\widetilde{V} (\xi_1, \xi_2).
\end{equation}
Take nonnegative functions $A, B$ on $\Z^n$ such that 
$A(\mu)=B(\mu)=0$ except for 
a finite number of $\mu \in \Z^n$, 
and define $f_j \in \calS(\R^n)$, $j=1,2$, by
\begin{align*}
f_1 (x)
=
\sum_{\nu_{1} \in \Z^n}
A(\nu_{1})
e^{i \nu_{1} \cdot x} 
(\calF^{-1}\varphi) (\lambda^{-1}x), \quad
f_2 (x)
=
\sum_{\nu_{2} \in \Z^n}
B(\nu_{2})
e^{i \nu_{2} \cdot x} 
(\calF^{-1}\varphi) (\lambda^{-1}x) 
\end{align*}
with $\lambda \in [1, \infty)$.

Since 
$\supp \varphi (\lambda (\cdot - \nu_1)) 
\subset \nu_1 + [-1/8, 1/8]^n$ for $\lambda \ge 1$, 
our choice of $\kappa$ implies  
\begin{equation*}
\kappa(\xi -k) \widehat{f_1}(\xi) 
=
\kappa (\xi - k)
\sum_{\nu_{1} \in \Z^n}
A(\nu_{1})\lambda^n \varphi(\lambda (\xi - \nu_{1})) 
=
A(k)\lambda^n \varphi(\lambda (\xi - k)).  
\end{equation*} 
Hence 
\[
\kappa(D-k) f_1(x)
=
A(k) e^{ix\cdot k} (\calF^{-1}\varphi) (\lambda^{-1}x) 
\]
and 
\begin{equation}\label{eqfjtest}
\begin{split}
&\|f_1\|_{W^{p_1,q_1}}
= 
\left\| \left\| 
\kappa(D-k) f_1(x)
 \right\|_{\ell^{q_1}_{k}} 
 \right\|_{L^{p_1}}
\\
&
=  
\left\| 
\left\|
A(k) e^{ix\cdot k} 
(\calF^{-1}\varphi) (\lambda^{-1}x) 
\right\|_{\ell^{q_1}_{k}} 
\right\|_{L^{p_1}_{x}} 
\approx 
\|A\|_{\ell^{q_1}(\Z^n)} 
\lambda^{n/p_1}.
\end{split}
\end{equation}
Similarly, 
\begin{align}\label{eqf2test}
\|f_2\|_{W^{p_2,q_2}}
\approx
\|B\|_{\ell^{q_1}(\Z^n)}
\lambda^{n/p_2}.
\end{align}
Since $\supp \varphi (\lambda (\cdot - \ell)) 
\subset \ell + [-1/8, 1/8]^n$ for $\lambda \ge 1$, 
our choice of 
$\widetilde\varphi$  
implies  
\begin{align*}
T_{\sigma}(f_1,f_2)(x)
&
=
\frac{1}{(2\pi)^n}
\int_{\R^n \times \R^n} 
e^{i x \cdot (\xi_1+ \xi_2)}
\sum_{ \nu_1, \nu_2 \in \Z^n }
V(\nu_1, \nu_2)
\\
&\qquad \times 
A(\nu_{1})B(\nu_{2}) 
\lambda^{n} \varphi (\lambda (\xi_1 - \nu_1))
\lambda^{n} \varphi (\lambda (\xi_2 - \nu_2))
\, 
d\xi_1 d\xi_2
\\
&
=
\sum_{ \nu_1, \nu_2 \in \Z^n }
V(\nu_1, \nu_2) A(\nu_{1})B(\nu_{2}) 
e^{i x \cdot (\nu_1 + \nu_2)}
(\calF^{-1}\varphi) (\lambda^{-1} x) ^2
\\
&
=\sum_{k}
d_k e^{i k \cdot x} 
 (\calF^{-1}\varphi) (\lambda^{-1} x) ^2,
\end{align*}
where
\[
d_k = 
\sum_{\nu_1,\nu_2\in\Z^n:\nu_1 + \nu_2 =k} 
V(\nu_1, \nu_2)
A(\nu_{1})B(\nu_{2}).
\]
Notice that 
$\supp (\varphi\ast\varphi) \subset 
[-1/4, 1/4]^n$.  
Hence, by the same reason as we obtained 
\eqref{eqfjtest}, 
we obtain 
\[
\kappa (D-k) 
\big[ T_{\sigma}(f_1,f_2) \big] (x)
=
d_k e^{i k \cdot x}
(\calF^{-1}\varphi) (\lambda^{-1} x)^2 
\]
and  
\begin{align}\label{eqTsigmatest}
\begin{split}
&\|T_{\sigma}(f_1,f_2)\|_{W^{p,q}}
= 
\left\| \left\| \kappa(D-k)
\left[ T_{\sigma}(f_1,f_2) \right](x) 
\right\|_{\ell^{q}_{k}} \right\|_{L^{p}_{x}}
\\
&=
\left\| \left\| 
d_k e^{ix\cdot k} (\calF^{-1}\varphi) (\lambda^{-1} x)^2 
\right\|_{\ell^{q}_{k}} 
\right\|_{L^{p}_{x}}
\approx
\|d_k\|_{\ell^{q}_{k}(\Z^n)} 
\lambda^{n/p}. 
\end{split}
\end{align}

The inequalities \eqref{eqinequality001} and \eqref{eqsigmatest} 
together with \eqref{eqfjtest}, \eqref{eqf2test}, and 
\eqref{eqTsigmatest} imply 
\[
\left\| 
d_k
\right\|_{\ell^q_k} 
\lambda^{n/p}
=
\left\| 
\sum_{\nu_1,\nu_2\in\Z^n:\nu_1 + \nu_2 =k} 
V(\nu_1, \nu_2)
A(\nu_{1})B(\nu_{2}) 
\right\|_{\ell^q_k} 
\lambda^{n/p} 
\lesssim 
\|A\|_{\ell^{q_1}} \| B \|_{\ell^{q_2}} 
\lambda^{n/p_1+ n/p_2}. 
\]
Thus, firstly taking $\lambda=1$ we obtain 
the inequality that implies 
$V \in \calB_{q_1,q_2,q}(\Z^n \times \Z^n)$.
Secondly, choosing 
$A\in \ell^{q_1}$ and $B\in \ell^{q_2}$ 
so that $d_k \not\equiv 0$, 
we obtain 
$\lambda^{n/p} = O ( 
\lambda^{n/p_1 + n/p_2})$ for $\lambda \ge 1$, 
which implies 
$1/p\le 1/p_1 + 1/p_2$. 
This completes the proof of the `only if' part of Theorem \ref{thmain-thm-WA}.  
\end{proof}

\subsection{Proof of Theorem \ref{thLp-calB}}
\label{proof-of-Lq-estimate}

Theorem \ref{thLp-calB} directly follows from the 
`if' part of  
Theorem \ref{thmain-thm-WA} and 
the embedding relations given in Proposition \ref{thWaembd}. 
In fact, 
suppose $p_1, p_2, p$ satisfy the assumptions of 
Theorem \ref{thLp-calB} and suppose 
\[
V \in \calBform_{\max (2, p_1^{\prime}), 
\max (2, p_2^{\prime}), \max (2, p) }
=
\calB_{
\max (2, p_1^{\prime}), 
\max (2, p_2^{\prime}), (\max (2, p))^{\prime} }. 
\]
Then 
the 
`if' part of 
Theorem \ref{thmain-thm-WA} 
implies 
\[
T_{\sigma}: W^{p_1,\max (2, p_1^{\prime})} 
\times W^{p_2, \max (2, p_2^{\prime})} 
\to W^{p,(\max (2,p))^{\prime}}. 
\] 
On the other hand, 
Proposition \ref{thWaembd} 
gives the embeddings 
\[ 
L^{p_1} \hookrightarrow W^{p_1, \max (2, p_1^{\prime})}, 
\quad   
L^{p_2} \hookrightarrow W^{p_2, \max (2, p_2^{\prime})}, 
\quad 
W^{p,(\max (2,p))^{\prime}} \hookrightarrow h^p.
\] 
Combining the above, 
we have $T_{\sigma}: L^{p_1}\times L^{p_2} \to h^{p}$. 
If $p_1=\infty$ (resp. $p_2=\infty$), 
then using the embedding 
$bmo \hookrightarrow W^{\infty,2}$ 
we can replace  
$L^{p_1}$ (resp. $L^{p_2}$) by $bmo$.

\subsection{Proof of Theorem \ref{thLp-diff-Lq}}
\label{proofofmain-cor_1}

To prove Theorem \ref{thLp-diff-Lq}, recall 
the 
Sobolev imbedding inequality with respect to $L^1$-norm: 
\begin{equation*}
|F(x)| 
\lesssim 
\sum_{|\beta|\le d} 
\int_{|x-z|<1} |\partial^{\beta} F (z)|\, dz   
\end{equation*} 
(see {\it e.g.\/} 
\cite[Chapter V, Section 6.4]{Stein_Diff}). 
Notice that this inequality obviousy implies 
\begin{equation}\label{eqSobolevL1}
|F(x)|\le 
c 
\sum_{|\beta|\le d} 
\int_{\R^d}
|\partial^{\beta}F (z)| \langle x-z \rangle ^{-d-1}\, dz, 
\quad x \in \R^d.   
\end{equation}

\begin{proof}[Proof of Theorem \ref{thLp-diff-Lq}] 
First 
we shall prove 
\eqref{eqLq-estimate-of-operator-norm} 
(the inequality 
with $L^q$-norm on the right hand side).  
Recall that $1\le q=q(p_1, p_2, p)\le 4$.  
Notice that Theorem \ref{thLp-calB} and 
Proposition \ref{thellq-ellqweak-calB} 
actually imply 
the 
following: 
there exist a constant 
$c$ and an integer $K^{\prime}$ 
depending only on 
$p_1, p_2, p$, and $n$ 
such that  
if $\sigma \in C^{\infty}((\R^n)^3)$, 
if $W$ is a moderate function on $\R^n \times \R^n$, 
and if 
\begin{equation}\label{eqdasigmaW}
\sum_{|\alpha|\le K^{\prime}}
\sup_{x\in \R^n}
|\partial_{x, \xi_1, \xi_2}^{\alpha} \sigma (x, \xi_1, \xi_2)|
\le 
W(\xi_1, \xi_2), 
\quad 
(\xi_1, \xi_2) \in \R^n \times \R^n, 
\end{equation}
then 
\begin{equation}\label{eqTsigmaW}
\|T_{\sigma}\|_{L^{p_1} \times L^{p_2} \to h^{p}}
\le 
c 
\| 
W\|_{L^{q} (\R^n \times \R^n)},  
\end{equation}
where 
$L^{p_1}$ (resp.\ $L^{p_2}$) can be 
replaced by $bmo$ 
if $p_1=\infty$ (resp.\ $p_2 = \infty$). 
By \eqref{eqSobolevL1}, 
we see that 
the inequality 
\eqref{eqdasigmaW} holds with 
\begin{equation}\label{eqW=dasigmaastangle}
W(\xi_1, \xi_2)
= c 
\sum_{\substack{
|\alpha|\le K^{\prime}
\\
|\beta|\le 2n
}}
\left(
\sup_{x \in \R^n} 
\left| 
\partial_{\xi_1, \xi_2}^{\beta}
\partial_{x, \xi_1, \xi_2}^{\alpha} 
\sigma (x, \xi_1, \xi_2)
\right|
\right)
\ast 
\langle (\xi_1, \xi_2)
\rangle ^{-2n-1}, 
\end{equation}
where $\ast$ denotes the convolution 
with respect to $(\xi_1, \xi_2)\in \R^n \times \R^n$. 
Here we assume 
$\sigma \not\equiv 0$ and 
the functions $\sup_{x}
|\partial_{\xi_1, \xi_2}^{\beta}
\partial_{x, \xi_1, \xi_2}^{\alpha} 
\sigma (x, \xi_1, \xi_2)|$ 
appearing in 
the right hand side of 
\eqref{eqW=dasigmaastangle} have 
finite $L^q$-norms. 
Then the function 
\eqref{eqW=dasigmaastangle} 
is a moderate function on $\R^n \times \R^n$ 
and 
\begin{equation*}
\|W\|_{L^{q} (\R^n \times \R^n)}
\lesssim 
\sum_{\substack{
|\alpha|\le K^{\prime}
\\
|\beta|\le 2n
}}
\left\|
\sup_{x \in \R^n} 
\left| 
\partial_{\xi_1, \xi_2}^{\beta}
\partial_{x, \xi_1, \xi_2}^{\alpha} 
\sigma (x, \xi_1, \xi_2)
\right|
\right\|_{L^q_{\xi_1, \xi_2} (\R^n \times\R^n)} 
\end{equation*}
by 
Minkowski's inequality (since $q\ge 1$).  
Hence \eqref{eqTsigmaW} yields the inequality 
\eqref{eqLq-estimate-of-operator-norm} with $K=K^{\prime}+2n$.

Next we prove 
\eqref{eqLq-estimate-of-operator-norm} 
with $L^q$-norm replaced by $L^{q, \infty}$-norm  
in the special cases 
(I), (II$\ast$), (III-1$\ast$), and (VI-1$\ast$).  
In these special cases,  
Theorem \ref{thLp-calB} and 
Proposition \ref{thellq-ellqweak-calB} imply that 
the estimate 
\eqref{eqTsigmaW} (with the same replacement of 
$L^{p_1}$ or $L^{p_2}$ in the cases $p_1$ or $p_2$ is infinity) 
holds with the $L^q$-norm of $W$ replaced by the 
$L^{q, \infty}$-norm. 
Notice that 
$2 \le q(p_1, p_2, p)\le 4$ in the special cases.  
Thus 
the proof  
is accomplished in almost the same way 
as above 
if we only use the fact that 
the inequality 
\[
\left\| 
V(\xi_1, \xi_2) \ast 
\langle (\xi_1, \xi_2)
\rangle ^{-2n-1}
\right\|_{L^{q, \infty} (\R^n \times \R^n)} 
\lesssim 
\left\| 
V(\xi_1, \xi_2)
\right\|_{L^{q, \infty} (\R^n \times \R^n)} 
\]
holds for $1<q< \infty$. 
This completes the proof of Theorem \ref{thLp-diff-Lq}.   
\end{proof}

\subsection{Proof of Corollary \ref{thLp-Ltildeq-GN}}

Under the assumptions 
of Corollary \ref{thLp-Ltildeq-GN}, 
Theorem \ref{thLp-diff-Lq} implies the 
estimate 
\begin{equation}\label{eqIneqfromCor17}
\| T_\sigma \|_{L^{p_1} \times L^{p_2} \to h^p}
\le c 
\sum_{|\alpha|\le K }
\left\| 
\sup_{x}
\left|
\partial^{\alpha}_{x, \xi_1, \xi_2}
\sigma(x,\xi_1,\xi_2) 
\right|
\right\|_{L^q_{\xi_1, \xi_2}(\R^{2n})}, 
\end{equation}
where 
$L^{p_1}$ (resp.\ $L^{p_2}$) can be 
replaced by $bmo$ 
if $p_1=\infty$ (resp.\ $p_2 = \infty$). 
Thus 
the claim of Corollary \ref{thLp-Ltildeq-GN} 
follows 
once we prove the inequality 
\begin{equation}\label{eqGagliardo-Nirenberg}
\begin{split}
&\sum_{|\alpha|\le K }
\left\| 
\sup_{x}
\left|
\partial^{\alpha}_{x, \xi_1, \xi_2}
\sigma(x,\xi_1,\xi_2) 
\right|
\right\|_{L^q_{\xi_1, \xi_2}(\R^{2n})}
\\
&
\le 
c 
\left\| 
\sup_{x}
| \sigma(x,\xi_1,\xi_2) |
\right\|_{L^{\widetilde{q}}_{\xi_1, \xi_2}(\R^{2n})}
^{\widetilde{q}/q}
\left( 
\sum_{|\alpha|\le N }
\left\|
\partial^{\alpha}_{x, \xi_1, \xi_2}
\sigma(x,\xi_1,\xi_2) 
\right\|_{L^{\infty}_{x, \xi_1, \xi_2}(\R^{3n})}
\right)^{1- \widetilde{q}/q}
\end{split}
\end{equation}
for $0<\widetilde{q}<q$ and for a sufficiently large integer $N>K$. 
This type of inequality 
is widely known as 
the Gagliardo--Nirenberg inequality. 
Notice, however, 
that the number 
$q=q(p_1, p_2, p)$ of Corollary \ref{thLp-Ltildeq-GN} 
(which always satisfies $1\le q \le 4$) 
can be equal to $1$. 
Thus we need \eqref{eqGagliardo-Nirenberg} including the case 
$\widetilde{q}<1=q$, 
which might not be well known. 
We shall give a proof 
of \eqref{eqGagliardo-Nirenberg} including the case 
$\widetilde{q}<1$ in Appendix.

A different method to prove 
Corollary \ref{thLp-Ltildeq-GN} that works 
in the case $1\le \widetilde{q}<q$ 
can be found in \cite[Section 4.6]{KMT-arxiv}.

\subsection{Sharpness of $q$ in 
Remark \ref{thHoldercase}. }
\label{sharpness-q}

In this subsection, 
we shall prove that in the case $1/p_1 + 1/p_2=1/p$ 
the number $q=q(p_1, p_2, p)$ in 
Theorem \ref{thLp-diff-Lq} 
is sharp in the sense that 
the estimate \eqref{eqLq-estimate-of-operator-norm} 
does not hold for 
$q> q(p_1, p_2, p)$.

For this purpose, we consider a multiplier of a special form. 
For a function $\Phi$ satisfying 
\begin{equation}\label{eqPhi}
\Phi \in \calS (\R^{2n}), 
\quad 
\supp \Phi \subset 
\{\zeta \in \R^{2n} :  |\zeta|\le 1/20\}
\end{equation}
and for a finite set $E \subset \Z^n \times \Z^n$,  
we define 
\begin{equation}\label{eqbumpmultiplier}
m_{E, \Phi}(\xi_1, \xi_2)
=
\sum_{(\nu_1,\nu_2)\in E} \Phi(\xi_1-\nu_1, \xi_2-\nu_2).
\end{equation}
Obviously this multiplier 
satisfies the estimate 
\[
\left| 
\partial_{\xi_1}^{\alpha_1} \partial_{\xi_2}^{\alpha_2}
m_{E, \Phi}(\xi_1, \xi_2)
\right|
\le C_{\Phi,\alpha_1, \alpha_2} 
\sum_{(\nu_1,\nu_2)\in E} 
\ichi_{Q}(\xi_1-\nu_1) 
\ichi_{Q}(\xi_2-\nu_2). 
\]
Thus 
Theorem 
\ref{thLp-diff-Lq} 
implies 
that the estimate 
\begin{equation}\label{eqestimateTmEPhi}
\| T_{m_{E,\Phi}} \|_{L^{p_1} \times L^{p_2} \to L^p} 
\le C_{\Phi} 
(\card E)^{1/q}
\end{equation}
holds with 
$q=q(p_1, p_2, p)$. 
The following proposition, 
which is essentially due to 
Buri\'ankov\'a--Grafakos--He--Honz\'ik 
\cite{BGHH}, 
asserts that even 
\eqref{eqestimateTmEPhi} 
for this simple operator 
$ T_{m_{E,\Phi}} $ 
does not hold for 
$q>q(p_1, p_2, p)$.

\begin{prop}[\cite{BGHH}]\label{thsharpofq}
Let $p_1,p_2 \in [1,\infty]$, $1/p_1+1/p_2=1/p$,
and $0 < q \le \infty$.
If 
the estimate \eqref{eqestimateTmEPhi}
holds for every 
$\Phi$ satisfying \eqref{eqPhi},  
then $q\le q(p_1, p_2, p)$. 
\end{prop}

\begin{proof}  
In fact, 
the claim of the proposition  
is proved in 
\cite[Section 6]{BGHH} 
in the case 
$1\le p_1, p_2 < \infty$. 
We shall prove that the claim also holds 
for $p_1=\infty$ or $p_2=\infty$. 
By symmetry, it is sufficient to consider the case 
$p_1=\infty$.

First consider the case $p_1=\infty$ and 
$p_2=p \in [1, 2]$. 
In this case $q(p_1, p_2, p)=2p_2$. 
Assume that the estimate 
\eqref{eqestimateTmEPhi} hold  
with $q=2p_2 + \epsilon$ for some positive $\epsilon$. 
By Theorem \ref{thLp-diff-Lq}, we know that the estimate 
$
\| T_{m_{E,\Phi}} \|_{
L^{2} \times L^{2} \to L^{1}
} 
\le C_{\Phi} 
(\card E)^{1/4}
$ 
holds. 
Hence by interpolation we have the estimate 
$
\| T_{m_{E,\Phi}} \|_{
L^{\widetilde{p_1}} \times L^{\widetilde{p_2}} \to L^{\widetilde{p}}
} 
\le C_{\Phi} 
(\card E)^{1/\widetilde{q}}
$
with 
\[
\left( {1}/{\widetilde{p_1}}, {1}/{\widetilde{p_2}}, 
{1}/{\widetilde{p}}, {1}/{\widetilde{q}} \right)
=
(1-\theta)
\left( {1}/{2}, {1}/{2}, {1}, {1}/{4} \right)
+\theta
\left( 0, {1}/{p_2}, {1}/{p_2}, {1}/{(2p_2+\epsilon)} \right), \quad
0 < \theta < 1.
\]
Notice that $2<\widetilde{p_1}<\infty$, 
$1 < \widetilde{p_2} \le 2$, and 
$\widetilde{q} > 2\widetilde{p_2}
=q(\widetilde{p_1}, \widetilde{p_2}, \widetilde{p})$. 
But this is impossible by the reuslt of 
\cite{BGHH} mentioned above. 
Therefore, for the estimate 
\eqref{eqestimateTmEPhi} 
the condition $q \le 2p_2 = q(\infty, p_2,p_2)$
is necessary.

Next consider the case $p_1=\infty$ and $p_2=p \in [2, \infty]$. 
In this case $q(p_1, p_2, p)=2p^{\prime}=2 (1- 1/p_2)^{-1}$.  
We use the same interpolation argument with
$\|  T_{m_{E, \Phi}} \|_{L^{a} \times L^{b} \to L^{2}}
\lesssim (\card E)^{1/4}$ 
for $2 \le a,b < \infty$ and $1/a+1/b=1/2$,
which is also proved in Theorem \ref{thLp-diff-Lq}.
The remaining argument is the same as above.
This completes the proof. 
\end{proof}

\appendix\section{}\label{sectionappendix}

Here we shall give a proof of the 
inequality \eqref{eqGagliardo-Nirenberg}. 
This inequality will be given in 
Lemma \ref{th04} below, 
which will cover a slightly more general case.

We begin with the following elementary lemma.

\begin{lem}\label{thVandermonde}
Let $N\ge 2$ be an integer,  
$\lambda, z\in \R$, and assume $\lambda \neq 0$.  
Then 
there exist 
polynomials $P_{k,j}$, 
$k, j \in \{ 0, 1, \dots, N-1\}$,  
of 
$N$ variables 
such that
\begin{equation}\label{eqPjk-01}
\lambda^{-\frac{N(N-1)}{2}} 
\sum_{k=0}^{N-1}
(\lambda k + z)^{j^{\prime}} 
P_{k,j} (z, \lambda + z, \dots, (N-1) \lambda  + z)
=
\begin{cases}
{1} & {\; {\text{if}} \;\; j=j^{\prime},  }\\ 
{0} & {\; {\text{if}} \;\; j\neq j^{\prime} }  
\end{cases}
\end{equation}
for 
$j, j^{\prime}\in \{0, 1, \dots, N-1\}$ 
and 
\begin{equation}\label{eqPjk-estimate}
\left| 
P_{k,j} (z, \lambda + z, \dots, (N-1) \lambda  + z)\right|
\le c 
(|z|+ |\lambda|)^{\frac{N(N-1)}{2}- j}
\end{equation}
with a constant $c$ that 
depends only on $N$. 
\end{lem}

\begin{proof}
We shall index the rows and columns of an $N\times N$ matrix by 
$0, 1, \dots, N-1$. 
Let $A$ be the $N\times N$ matrix with 
the $(j,k)$ entry equal to 
$(\lambda k + z)^{j} $, 
$j, k\in \{0, 1, \dots, N-1\}$. 
The determinant of $A$ is 
Vandermonde's determinant 
and is given by 
\[
\det A = 
\prod_{N-1\ge k>k^{\prime}\ge 0}
( (\lambda k + z) - (\lambda k^{\prime}  + z) )
=
\lambda^{\frac{N(N-1)}{2}} 
\prod_{N-1\ge k>k^{\prime}\ge 0}
( k- k^{\prime} ). 
\]
Thus 
$A$ has the inverse matrix when $\lambda \neq 0$.  
Let $A_{j,k}$ be the 
cofactor 
of $(\lambda k + z)^{j}$ 
in $A$. 
Then the  
$(k,j)$ entry of the inverse matrix $A^{-1}$ is given by 
\begin{equation*}
(\det A)^{-1} A_{j,k}
=
\lambda^{-\frac{N(N-1)}{2}}
\bigg( \prod_{N-1\ge k>k^{\prime}\ge 0}
( k- k^{\prime} )
\bigg)^{-1}
A_{j,k}.  
\end{equation*}
Observe that the cofactor 
$A_{j,k}$ is 
a polynomial in 
the $N$ variables 
$z, \lambda  + z, \dots, (N-1)\lambda + z$
that is homogeneous of degree 	
$\frac{N(N-1)}{2}- j$. 
Thus 
\begin{align*}
&P_{k,j}(z, \lambda  + z, \dots, (N-1)\lambda + z)
\\
&=
 \bigg( \prod_{N-1\ge k>k^{\prime}\ge 0}
( k- k^{\prime} )
\bigg)^{-1}
A_{j,k}(z, \lambda  + z, \dots, (N-1)\lambda + z)
\end{align*}
satisfies \eqref{eqPjk-01}. 
The estimate 
\eqref{eqPjk-estimate} 
is obvious since 
$P_{k, j}$ is a homogeneous polynomial of degree 
$\frac{N(N-1)}{2}- j$. 
\end{proof}

Let $d, N\in \N$ and $N\ge 2$. 
For 
multi-indices $\beta, \gamma \in (\N_0)^d$ satisfying 
\begin{equation*}
0\le \beta_j, \gamma_j \le N-1, 
\quad 
j=1, \dots, d, 
\end{equation*}
and 
for 
$\lambda \in \R \setminus \{0\}$ and 
$z=(z_1, \dots, z_d)\in \R^d$, 
we define 
\[
\widetilde{P}_{\beta, \gamma}
(N, \lambda, z) 
=
\prod_{j=1}^{d} 
P_{\beta_j, \gamma_j}(z_j, \lambda + z_j, \dots, (N-1) \lambda + z_j), 
\]
where 
$P_{k,j}$ are the polynomials given in 
Lemma \ref{thVandermonde}.

Then Lemma \ref{thVandermonde} implies the following.

\begin{lem}\label{thd-Vandermonde}
Let $d, N\in \N$ and $N\ge 2$. 
Then for multi-indices 
$\alpha, \beta, \gamma\in (\N_0)^{d}$ 
satisfying 
$0\le \alpha_j, \beta_j, \gamma_j \le N-1$ 
($j=1, \dots, d$)
and for 
$\lambda \in \R\setminus \{0\}$ and 
$z\in \R^d$, 
we have 
\begin{equation}\label{eqtildePab-01}
\lambda^{-\frac{N(N-1)d}{2}} 
\sum_{\beta}
(\lambda \beta + z)^{\alpha} 
\widetilde{P}_{\beta, \gamma}
(N, \lambda, z) 
=
\begin{cases}
{1} & {\; {\text{if}} \;\; \alpha=\gamma, }\\ 
{0} & {\; \text{if} \;\; \alpha \neq \gamma}
\end{cases}
\end{equation}
and 
\begin{equation}\label{eqtildePab-estimate}
\left| 
\widetilde{P}_{\beta, \gamma}
(N, \lambda, z) \right|
\le 
c (|z|+ |\lambda|)^{\frac{N(N-1)d}{2}-|\gamma|}
\end{equation}
with 
a constant 
$c$ depending only on 
$d$ and $N$. 
\end{lem}

We write $M_{1}$ to denote the usual Hardy-Littlewood maximal operator 
and define 
$M_{\epsilon}(f) = M_{1}(|f|^{\epsilon})^{1/\epsilon}$ for $\epsilon>0$.

\begin{lem}\label{thlogconvexMeM1} 
Let $d, N\in \N$ and $N\ge 2$. 
Let $f\in C^{N}(\R^d)$ and let $\gamma \in (\N_0)^d$ 
be a multi-index satisfying 
$0<|\gamma|<N$. 
Then, for each $0< \epsilon \le 1$, 
there exists a constant 
$c$ depending only on 
$d, N, \epsilon$, 
such that the inequality 
\begin{equation*}
\left| \partial^{\gamma} f (y) \right|
\le 
c 
\left[ M_{\epsilon} (f) (y) \right]^{(N-|\gamma|)/N}
\bigg[ 
\sum_{|\alpha|=N} 
M_{1} ( |\partial^{\alpha}f| ) (y) 
\bigg]^{|\gamma|/N}
\end{equation*}
holds 
for 
all $y\in \R^d$ satisfying 
$M_{\epsilon} (f) (y) <\infty$ and 
$\sum_{|\alpha|=N} M_{1} (  |\partial^{\alpha}f|  ) (y) <\infty$. 
\end{lem}

\begin{proof}
Take an arbitrary $\lambda \in (0, \infty)$. 
Taylor's formula gives 
\begin{align*}
&
\sum_{|\alpha|\le N-1} 
\frac{\partial^{\alpha}f (y)}{\alpha !} z^{\alpha} 
=f(y+z) 
-
R_{N}(y,z), 
\\
&
R_{N}(y,z)
=
\sum_{|\alpha|=N} 
\int_{0}^{1} N (1-\theta)^{N-1} 
\frac{(\partial^{\alpha}f)
\big(y + \theta z\big)}{\alpha !}
z^{\alpha}\, d\theta. 
\end{align*}
In this formula, replace 
$z$ by $\lambda \beta + z$, 
multiply 
$\lambda^{-\frac{N(N-1)d}{2}} \widetilde{P}_{\beta,\gamma} (N, \lambda, z)$, 
and take sum over $\beta$. 
Then 
\eqref{eqtildePab-01} yields 
\begin{align*}
\frac{\partial^{\gamma}f(y)}{\gamma !}
&
=
\lambda^{-\frac{N(N-1)d}{2}} 
\sum_{\beta}
f(y+ \lambda \beta + z) 
\widetilde{P}_{\beta,\gamma} (N, \lambda, z)
\\
&\quad 
-\lambda^{-\frac{N(N-1)d}{2}}
\sum_{\beta}
R_{N}(y, \lambda \beta + z) 
\widetilde{P}_{\beta,\gamma} (N, \lambda, z). 
\end{align*}
If $|z|<\lambda$, then 
the estimate 
\eqref{eqtildePab-estimate} implies 
\begin{align*}
|\partial^{\gamma}f(y)| 
\lesssim 
\lambda^{-|\gamma |}  
\sum_{\beta}
| f(y+ \lambda \beta + z) | 
+ \lambda^{-|\gamma| }
\sum_{\beta}
|R_{N}(y, \lambda \beta + z)|.  
\end{align*}
Let $0<\epsilon \le 1$. 
Taking average 
over $z$ in the range $|z|<\lambda$, 
we obtain 
\begin{equation*}
\begin{split}
|\partial^{\gamma} f (y)|
& \lesssim 
\lambda^{-|\gamma|} 
\sum_{\beta}
\bigg( \lambda^{-d}
\int_{|z|<\lambda} 
|f(y+\lambda \beta + z)|^{\epsilon}\, dz  
\bigg)^{1/\epsilon}
\\
&\quad 
 + 
\lambda^{-|\gamma|}
\sum_{\beta}
\bigg( \lambda^{-d} 
\int_{|z|<\lambda} 
\left| R_{N}(y, \lambda \beta + z)\right|^{\epsilon}\, 
dz \bigg)^{1/\epsilon}
\\
&
:= \mathrm{I} + \mathrm{II}. 
\end{split}
\end{equation*}
For the first term, we have 
\begin{equation*}
\mathrm{I}
\lesssim 
\lambda^{-|\gamma|} 
\bigg( \lambda^{-d}
\int_{|w|\lesssim \lambda} 
|f(y+w)|^{\epsilon}\, dw 
\bigg)^{1/\epsilon}
\lesssim 
\lambda^{-|\gamma|} 
M_{\epsilon}(f)(y). 
\end{equation*}
For the second term, 
we use H\"older's inequality (recall $\epsilon \le 1$) 
and make a change of variables 
$\theta (\lambda \beta + z)=w$ to obtain 
\begin{align*}
\mathrm{II} 
&
\lesssim 
\lambda^{-|\gamma|}
\sum_{\beta}
\lambda^{-d} 
\int_{|z|<\lambda} 
| R_{N}(y, \lambda \beta + z)|\, 
dz 
\\
&
\lesssim 
\lambda^{-|\gamma|}
\sum_{\beta}
\sum_{|\alpha|=N}
\lambda^{-d} 
\iint_{\substack{ 
|z|<\lambda
\\
0<\theta<1
}} 
\left| 
(\partial^{\alpha}f)(y+ \theta (\lambda \beta + z)) \right|
\lambda^{N}
\, d\theta dz 
\\
&
\le 
\lambda^{-|\gamma|}
\sum_{|\alpha|=N}
\lambda^{-d}
\iint_{\substack{ 
|w|\lesssim \theta \lambda
\\
0<\theta<1
}} 
\left| 
(\partial^{\alpha}f) 
(y+ w) \right|
\lambda^{N} 
\theta^{-d}
\, 
d\theta dw
\\
&
\lesssim 
\lambda^{N-|\gamma|}
\sum_{|\alpha|=N} 
M_{1}(|\partial^{\alpha}f|)(y). 
\end{align*}
Thus 
\[
|\partial^{\gamma}f(y)|
\lesssim 
\lambda^{-|\gamma|}
M_{\epsilon}(f)(y)
+
\lambda^{N-|\gamma|}
\sum_{|\alpha|=N}
M_{1}( |\partial^{\alpha}f|)(y).   
\] 
The implicit constant in this inequality does not 
depend on $\lambda > 0$. 
Hence, if 
$M_{\epsilon}(f)(y)<\infty$ and 
$\sum_{|\alpha|=N} M_{1}(|\partial^{\alpha}f|)(y)<\infty$,  
then 
taking infimum over $\lambda>0$ we 
obtain the inequality claimed in the lemma. 
\end{proof}

Now the following lemma includes 
the inequality \eqref{eqGagliardo-Nirenberg} as a special case.

\begin{lem}\label{th04}
Let $n, d, K\in \N$. 
Let $1<r \le \infty$, 
$0<\widetilde{q} < q< r $, 
and let 
$\theta$ be the number defined by 
$1/q = (1-\theta)/\widetilde{q} + \theta/r$. 
Then there exist a positive integer $N> K$ 
and a constant 
$c$ 
such that 
\begin{equation*}
\begin{split}
&\bigg\|
\sum_{
|\alpha|\le K}
\sup_{x \in \R^n} 
\left|
\partial_{x,y}^{\alpha}
f(x,y)
\right|
\bigg\|_{L^{q}_{y}(\R^d)}
\\
&
\le 
c
\left\|
\sup_{x\in \R^n} 
|f(x,y)|
\right\|_{L^{\widetilde{q}}_{y}(\R^d)}
^{1-\theta}
\bigg\|  
\sum_{|\alpha|\le N } 
\sup_{x \in \R^n} 
\left| 
\partial_{x,y}^{\alpha}
f(x,y) \right|
\bigg\|_{L^{r}_{y}(\R^d)}
^{\theta}
\end{split}
\end{equation*}
for all 
$f \in C^{\infty}(\R^n \times \R^d)$. 
The 
integer $N$ can be taken depending only on 
$K, q, \widetilde{q}$, and $r$;  
the constant $c$ depends 
only on 
$n, d, K, q, \widetilde{q}$, and $r$. 
\end{lem}

\begin{proof}
If we 
prove the inequality 
for 
$f\in C^{\infty}(\R^n \times \R^d)$ with compact support, 
then we can 
easily derive  
the inequality for general $f$ 
by a limiting argument. 
Thus we assume 
$f\in C^{\infty}(\R^n \times \R^d)$ has compact support. 
We shall prove the inequality 
for an $N\in \N$ satisfying 
$K/N\le \theta$.

Take an arbitrary 
$\alpha \in (\N_0)^{n+d}$ satisfying  
$|\alpha|\le K$. 
We write 
$\theta_{\ast}= |\alpha|/N$. 
Take an $\epsilon \in (0,1]$. 
By Lemma \ref{thlogconvexMeM1}, we 
have 
\begin{equation*}
\left| \partial_{x,y}^{\alpha} f (x, y) \right|
\lesssim  
\left[ M_{\epsilon} (f) (x, y) \right]^{1 - \theta_{\ast}}
\bigg[ 
\sum_{|\alpha| =N} 
M_{1} ( 
|\partial_{x,y}^{\alpha} f |  ) (x, y) 
\bigg]^{\theta_{\ast}}  
\end{equation*}
(notice that this inequality obviously holds if $\alpha=0$).   
Taking $\sup$ over $x\in \R^n$, we have 
\begin{equation}\label{eqpointwise}
\sup_{x} 
\left| \partial_{x, y}^{\alpha} f (x, y) \right|
\lesssim  
\left[ M_{\epsilon} 
\left( 
\sup_{x} |f(x, y)| 
\right) 
\right]^{1 - \theta_{\ast}}
\bigg[ 
\sum_{|\alpha|=N} 
M_{1} 
\bigg( 
\sup_{x} |\partial_{x, y}^{\alpha} f (x,y)|  \bigg)  
\bigg]^{\theta_{\ast}}. 
\end{equation}
Since 
$\theta_{\ast}\le \theta$ by 
our choice of 
$N$, 
we have 
\begin{align*}
&(\text{the right hand side of \eqref{eqpointwise}})
\\
&
=
\left[ M_{\epsilon} 
\left( 
\sup_{x} |f(x, y)| 
\right) 
\right]^{1 - \theta}
\left[ M_{\epsilon} 
\left( 
\sup_{x} |f(x, y)| 
\right) 
\right]^{\theta - \theta_{\ast}}
\bigg[ 
\sum_{|\alpha|=N} 
M_{1} \bigg( 
\sup_{x}
|\partial_{x,y}^{\alpha} f(x,y) |  \bigg) 
\bigg]^{\theta_{\ast}}
\\
&
\le
\left[ M_{\epsilon} 
\left( 
\sup_{x} |f(x, y)| 
\right) 
\right]^{1 - \theta}
\bigg[ 
\sum_{|\alpha|\le N} 
M_{1} \bigg( 
\sup_{x} 
|\partial_{x,y}^{\alpha} f(x,y)|  \bigg) 
\bigg]^{\theta}. 
\end{align*}
The maximal  operator $M_{1}$ is bounded in 
$L^{r}$ since $r>1$. 
We choose 
$\epsilon$ so that 
$\epsilon < \widetilde{q}$; 
then 
$M_{\epsilon}$ is bounded in $L^{\widetilde{q}}$. 
Hence 
the above 
inequalities combined with 
H\"older's inequality and the boundedness of 
$M_{\epsilon}$ and $M_{1}$ 
imply  
\begin{align*}
&\bigg\| \sup_{x} 
| \partial_{x,y}^{\alpha} f (x, y) | 
\bigg\|_{L^q_y}
\\
&
\lesssim  
\bigg\| 
M_{\epsilon} 
\left( 
\sup_{x} |f(x, y)| 
\right) 
\bigg\|_{L^{\widetilde{q}}_y}
^{1 - \theta}
\bigg\|  
\sum_{|\alpha|\le N} 
M_{1} \bigg( 
\sup_{x} 
|\partial_{x,y}^{\alpha} f(x,y)|  \bigg) 
\bigg\|_{L^{r}_y}
^{\theta}
\\
&
\lesssim 
\bigg\| 
\sup_{x} |f(x, y)| 
\bigg\|_{L^{\widetilde{q}}_y}^{1 - \theta}
\bigg\|  
\sum_{|\alpha|\le N} 
\sup_{x} 
|\partial_{x,y}^{\alpha} f(x,y)|  
\bigg\|_{L^{r}_y}
^{\theta}, 
\end{align*}
which is the desired inequality. 
\end{proof}



\end{document}